\documentclass[a4paper,reqno,11pt]{amsart}
\usepackage{amsfonts}
\usepackage{amsaddr}
\usepackage{hyperref}
\usepackage[left=3cm,right=3cm,top=3cm,bottom=2.7cm,headsep=0.7cm]{geometry}
\usepackage[T1]{fontenc} 
\usepackage[utf8]{inputenc}
\usepackage{amsmath,amssymb,amsthm,mathrsfs,esint}
\usepackage{mathtools}
\usepackage[titletoc]{appendix}

\newcommand{\R}{{\mathbb R}}
\newcommand{\C}{{\mathbb C}}
\newcommand{\N}{{\mathbb N}}
\newcommand{\Z}{{\mathbb Z}}
\newcommand{\m}{{\textbf{m}}}

\newcommand{\Rn}{{\R}^n}

\newcommand{\dx}{\, \mathrm{d} x}
\newcommand{\dz}{\, \mathrm{d} z}

\newcommand{\hn}{\mathcal{H}^{n-1}}

\newcommand{\Sn}{{\mathbb{S}^{n-1}}}
\newcommand{\ol}{\overline}

\newcommand{\wt}{\widetilde}

\newcommand{\sm}{\setminus}
\newcommand{\Mnn}{{\mathbb{M}^{n\times n}_{sym}}}
\newcommand{\dod}{{\partial_D \Omega}}
\newcommand{\don}{{\partial_N \Omega}}
\newcommand{\dom}{{\partial \Omega}}
\newcommand{\weak}{\rightharpoonup}

\newcommand{\mres}{\mathbin{\vrule height 1.6ex depth 0pt width
0.13ex\vrule height 0.13ex depth 0pt width 1.3ex}}

\DeclareMathOperator*{\aplim}{ap\,lim}

\theoremstyle{plain}
\begingroup
\theoremstyle{plain}
\newtheorem{theorem}{Theorem}[section]
\newtheorem{corollary}[theorem]{Corollary}
\newtheorem{proposition}[theorem]{Proposition}
\newtheorem{lemma}[theorem]{Lemma}
\theoremstyle{definition}
\newtheorem{definition}[theorem]{Definition}
\theoremstyle{remark}
\newtheorem{remark}[theorem]{Remark}

\endgroup

\numberwithin{equation}{section}
\title[Strong solutions to the Dirichlet problem for the Griffith energy]{Existence of strong solutions to the Dirichlet problem for the Griffith energy}
\author{Antonin Chambolle \and Vito Crismale}
\address{CMAP, \'Ecole Polytechnique, CNRS, 91128 Palaiseau Cedex, France}
\email[Antonin Chambolle]{antonin.chambolle@cmap.polytechnique.fr}
\email[Vito Crismale]{vito.crismale@polytechnique.edu}

\begin{document}
\begin{abstract}
  In this paper we continue the study of the Griffith brittle fracture energy minimisation under Dirichlet
  boundary conditions, suggested  by Francfort and Marigo in 1998~\cite{FraMar98}. In a recent
  paper~\cite{CC18}  we proved the existence of weak minimisers  of the problem. Now we show that these minimisers are indeed strong solutions, namely their jump set
  is closed and they are smooth away from the jump set and continuous up to the Dirichlet boundary.
  This is obtained by extending up to the boundary the recent regularity results of
  Conti, Focardi and Iurlano~\cite{CFI17DCL} and Chambolle, Conti, Iurlano~\cite{ChaConIur17}.
\end{abstract}


\maketitle

\tableofcontents

\section{Introduction}

We prove that the minimisation problem for the Griffith brittle fracture energy \cite{Griffith} under Dirichlet boundary conditions admits so-called \emph{strong solutions},  if the Dirichlet part of the boundary is of class $C^1$.  Given an  open  bounded  
\emph{reference configuration} $\Omega \subset \Rn$, $n\geq 2$, with  $\dom$ of null $n$-dimensional Lebesgue measure,  $\dod \subset \dom$ (relatively) open,  $K \subset \Omega \cup \dod$ (relatively) closed, and a boundary datum $u_0 \in W^{1,\infty}(\Rn; \Rn)$, a strong solution $(\ol u, \ol \Gamma)$ minimises
\begin{equation}\label{enGriffith}\tag{G}
(u, \Gamma) \mapsto \hspace{-1em} \int \limits_{\Omega'\sm (\Gamma \cup K) } \hspace{-1em}\C e(u) \colon e(u) \dx + 2\beta\, \hn\big((\Gamma \sm K) \cap \Omega'\big)
\end{equation}
in the class
\begin{equation*}
\mathcal{A}:= \{ (u, \Gamma) \colon \Gamma \text{ closed}, \,  u=u_0 \text{ in } \Omega' \sm (\Omega\cup \dod),\,  u \in C^1(\Omega \sm (\Gamma \cup K); \Rn) \cap C(\Omega' \sm (\Gamma \cup K); \Rn) \}\,,
\end{equation*}
 where $\Omega'$ is open with $\Omega \subset \Omega'$, $\Omega' \cap \dom = \dod$, and $\mathrm{diam}\, \Omega' \leq 2 \, \mathrm{diam}\, \Omega$ (in the following $\don:= \dom \sm \dod$).
In \eqref{enGriffith}, $e(u)$ is the symmetrised gradient of $u$, expressing the \emph{infinitesimal elastic strain}, $\C$ is the \emph{Cauchy stress tensor}, satisfying
\begin{equation}\label{eq:hpC}
\C(\xi-\xi^T)=0 \quad \text{and} \quad \C \xi \cdot \xi \geq c_0\, |\xi+\xi^T|^2 \quad\text{for all }\xi \in \Mnn\,,
\end{equation}
and $\beta>0$ represents the \emph{toughness} of the material. 
The variational approach to fracture by Francfort and Marigo \cite{FraMar98} is based, for quasistatic evolutions in brittle fracture, on successive minimisation of \eqref{enGriffith}, where $K$ is the crack set computed at the previous step  (it is not restrictive to assume $K=\emptyset$,
 replacing $\Omega$ with the new reference configuration $\Omega \sm K$).

The strategy to prove existence of strong minimisers is in the spirit of De Giorgi's approach to Mumford-Shah functional \cite{MumSha}: study the existence of minimisers for a \emph{weak formulation} where $\Gamma$ is replaced by the intrinsic jump set $J_u$ of $u$, for $u$ in a suitable admissible class such that $J_u$ is countably $(\hn, n{-}1)$ rectifiable, and prove that any weak minimiser corresponds to a strong one, that is $J_u$ is closed and $u$ is of class  $C^1$ outside  $J_u$. (In fact, when we say $J_u$ closed we mean always essentially closed, that is closed up to a $\hn$-negligible set,  and $u$ is indeed of class $C^\infty$ outside $J_u$.)

For Mumford-Shah functional, this has been realised by Ambrosio \cite{Amb89UMI, Amb90GSBV, AmbNew95} and De Giorgi, Carriero, Leaci \cite{DeGCarLea} (see also \cite{DMMorSol92, MadSol01} for a different approach, and e.g.\ \cite{FonFus97, David05Libro} and the references in \cite{LemReviewMS} for other regularity results). For the case of brittle fracture in the antiplane shear setting, that formally corresponds to remove the fidelity term in the Mumford-Shah functional and consider the Dirichlet minimum problem, or for brittle fracture with finite strain elasticity, the regularity up to the boundary is proven by Babadjian and Giacomini in \cite{BabGia14}, which inspired the present work (we refer to e.g.\ \cite{FraLar03, DMToa02, DMFraToa07, DMLaz10} for existence of quasistatic evolutions for antiplane brittle fracture or brittle fracture with finite strain elasticity).

The existence of weak minimiser of \eqref{enGriffith} has been proven recently in \cite{CC18} with a general compactness and lower semicontinuity result for $GSBD$, the space introduced by Dal Maso in \cite{DM13} to include all the displacements with finite Griffith energy. Weak minimisers of \eqref{enGriffith} were known to exist under simplifying assumptions: an \emph{a priori} $L^\infty$ assumption on displacement, for \cite{BelCosDM98} in the $SBD$ space \cite{AmbCosDM97}; the connectedness of $\Gamma$ in \cite{Cha03}; a mild fidelity term in \cite{DM13}; in dimension 2 in \cite{FriSol16} (we mention also \cite{Fri19} and the approximations in \cite{Iur14, FriPWKorn, CFI17Density, CC17, Fri17ARMA, ChaConFra18ARMA, Cha04, Cri19}).

The regularity result analogous to De Giorgi, Carriero, Leaci \cite{DeGCarLea} is obtained in \cite{CFI17DCL} in dimension 2  (for more general energies)  and in \cite{ChaConIur17} in general dimension (see also \cite{CFI18nota}), ensuring closedness in $\Omega$ of the jump set of weak minimisers, thus existence of strong minimisers for the problem with fidelity term. The present work extends this regularity up to the boundary (that is in $\Omega'$), assuming $\dod$ of class $C^1$, in the main results Theorem~\ref{teo:regDirLoc} and Corollary~\ref{cor:1106181816}. This shows  existence of strong minimisers for \eqref{enGriffith}. 

As in the other regularity results, the key point (Theorem~\ref{teo:Teor3.4BG}) is a density lower bound for the jump set of minimisers, that now holds for all balls centered on a point in $J_u$,  contained in the enlarged domain $\Omega'$, with radius small enough  (and at a small security distance from $\partial(\dod)$ if $\Omega$ is not of class $C^1$, see also Remark~\ref{rem:sceltaeta}). This is the analogous of \cite[Theorem~3.4]{BabGia14}, while in \cite{DeGCarLea, CFI17DCL, ChaConIur17} the density lower bound is proven only for balls contained in $\Omega$.

Following the usual scheme by contradiction, we are led to prove a decay estimate for the Griffith energy of local minimisers in balls with vanishing radius and vanishing $(n{-}1)$-dimensional density of the jump set. If the balls are contained in $\Omega$ this is done as in \cite{ChaConIur17}, showing that (local) quasi-minimisers with vanishing jump set on the ball $B(0,1)$, obtained by blow-up, converge to a local minimiser for the bulk energy. 

In the case where the balls intersect $\Omega' \sm \Omega$ we have a sequence of quasi-minimisers for a problem with a prescribed displacement outside $\Omega$: we then modify (Theorem~\ref{thm:Teor4}) the compactness result for functions with vanishing jump \cite[Theorem~4]{ChaConIur17} to include the case of a prescribed value somewhere. This passes also through an approximation 
of $GSBD$ displacements with small jump set and a prescribed value in a subdomain $D$, through functions keeping the same value both in $D$ and near the boundary, and smooth in the interior (Theorem~\ref{teo:teor3CCI}). The proof of the compactness result is done in the spirit of the corresponding \cite[Theorem~3]{ChaConIur17}, employing 
a Korn-Poincaré-type estimate by \cite{CCF16}.

The regularity of minimisers with prescribed value on a subdomain is obtained in Theorem~\ref{teo:ellreg2} by adapting a regularity result for solution to elliptic systems in \cite{McLean}. Differently from the analogous \cite[Theorem~3.8]{BabGia14}, we only consider the case of quadratic growth for the bulk energy. Indeed, even in the unconstrained minimisation problem, the desired regularity { seems} by now available only for quadratic growth in general dimension, cf.\ \cite{CFI18nota, ChaConIur17}. This is the only point that prevents to obtain a more general regularity result in the case where the bulk energy has growth $p>1$ in $e(u)$ for $|e(u)|$ large,  as in \cite{CFI17DCL}. 

 The decay estimate guarantees a uniform density lower bound for $\hn$ in balls centered in $\ol{J^*_u}$, with radii less than a uniform value $\varrho_0$, where $J^*_u \subset J_u$ is given by the jump points of full density with respect to $\hn$ (Theorem~\ref{teo:Teor3.4BG}). 
 This implies that $J^*_u$ is essentially closed, then by  elliptic  regularity we get that $u$ is of class  $C^\infty$  in { $\Omega \sm \ol{J^*_u}$ and $\ol{J_u}$ coincides with $\ol{J^*_u}$ in $\Omega$}.
 {  The continuity of the minimiser $u$ up to $\dod \sm J^*_u$ is derived from the fact that by minimality, $\|e(u)\|^2_{L^2(B_\varrho(x))}$ is controlled by $\varrho^{n-1}$ for any $B_\varrho(x) \subset \Omega'$.
   Arguing as in Campanato's theorem (with  infinitesimal rigid motions in place of averages on balls)
   we deduce that $|u_0(x_0)-u(x)|\leq C \sqrt{|x_0-x|}$ near any $x_0 \in \dod \sm \ol{J^*_u}$. In particular,
   it follows that $J_u\subset \ol{J^*_u}$ in $\Omega \cup \dod$. }

As a  concluding  remark, we observe that in dimension~2, the authors of~\cite{BabGia14} prove the existence of a strong quasistatic evolution, namely minimising the antiplane version of \eqref{enGriffith} with respect to its own (closed) jump set at any time $t$.
The starting point is therein the existence result \cite{FraLar03}, that has been recently extended
{ to planar elasticity} by Friedrich and Solombrino in \cite{FriSol16} in dimension~2.
 In the present context it is immediate to combine our density lower bound with the geometrical 2d argument in \cite[Proposition~5.5]{BabGia14} to get that the sequence of piecewise-constant in time evolutions $u_k$ in the Francfort-Marigo approach (obtaining by dividing the given time interval $[0,T]$ by $k{+}1$ nodes $t_k^i= i \frac{T}{k}$ and interpolating in time the solutions to the incremental minimum problems in the nodes) satisfies a density lower bound uniform in time and in $k$.
However, the improvement of the evolution in \cite{FriSol16} seems  
delicate. Indeed, the tool of $\sigma_p$-convergence, developed in \cite{DMFraToa07} and crucial in \cite{BabGia14}, 
is not directly applicable now, as we do not work in $SBV^p$.

\section{Notation and preliminaries}\label{sec:Sec2}

We denote by $\mathcal{L}^n$ and $\mathcal{H}^k$ the $n$-dimensional Lebesgue measure and the $k$-dimensional Hausdorff measure. For any locally compact subset $B$ of $\Rn$, the space of bounded $\R^m$-valued Radon measures on $B$ is indicated by $\mathcal{M}_b(B;\R^m)$  (the space of $\R^m$-valued Radon measures on $B$ is denoted by $\mathcal{M}(B;\R^m)$).  For $m=1$ we write $\mathcal{M}_b(B)$ for $\mathcal{M}_b(B;\R)$,  $\mathcal{M}(B)$ for $\mathcal{M}(B;\R)$,  and $\mathcal{M}^+_b(B)$ for the subspace of positive measures of $\mathcal{M}_b(B)$. For every $\mu \in \mathcal{M}_b(B;\R^m)$, its total variation is denoted by $|\mu|(B)$.
We write $\chi_E$ for the indicator function of any $E\subset \R^n$, which is 1 on $E$ and 0 otherwise, and $A_1 \Subset A_2$ for two open sets $A_1$, $A_2$ such that $\ol A_1 \subset A_2$.
For every $x\in \Rn$ and $\varrho>0$, $B_\varrho(x)$ is the open ball with center $x$ and radius $\varrho$.  We usually write in the following $B_\varrho$ for $B_\varrho(0)$. 
\par
\medskip
\paragraph{\bf Function spaces.}
Let $U\subset \Rn$ be open and bounded.  For any $\mathcal{L}^n$-measurable function $v\colon U\to \R^m$ the \emph{approximate jump set} $J_v$ is the set of points $x\in U$ for which there exist $a$, $b\in \R^m$, with $a \neq b$, and $\nu\in \Sn$ such that  (see e.g.\ \cite[Section~3.6]{AFP}) 
\begin{equation*}
\aplim\limits_{(y-x)\cdot \nu>0,\, y \to x} v(y)=a\quad\text{and}\quad \aplim\limits_{(y-x)\cdot \nu<0, \, y \to x} v(y)=b\,.
\end{equation*} 
A function $v\in L^1(U)$ is a \emph{function of bounded variation} on $U$ ($v\in BV(U)$), if $\mathrm{D}_i v\in \mathcal{M}_b(U)$ for $i=1,\dots,n$, where $\mathrm{D}v=(\mathrm{D}_1 v,\dots, \mathrm{D}_n v)$ is its distributional gradient. A vector-valued function $v\colon U\to \R^m$ is in $BV(U;\R^m)$ if $v_j\in BV(U)$ for every $j=1,\dots, m$.
The space $BV_{\mathrm{loc}}(U)$ is the space of $v\in L^1_{\mathrm{loc}}(U)$ such that  $\mathrm{D}_i v\in \mathcal{M}(U)$  for $i=1,\dots,n$.

A function $v\in L^1(U;\Rn)$ belongs to the space of \emph{functions of bounded deformation} $BD(U)$ if its distributional symmetric gradient $\mathrm{E}v$ belongs to  $\mathcal{M}_b(U;\Mnn)$. 
It is well known (see \cite{AmbCosDM97, Tem}) that for $v\in BD(U)$, $J_v$ is countably $(\hn, n-1)$ rectifiable, and that
\begin{equation}
\mathrm{E}v=\mathrm{E}^a v+ \mathrm{E}^c v + \mathrm{E}^j v\,,
\end{equation}
where $\mathrm{E}^a v$ is absolutely continuous with respect to $\mathcal{L}^n$, { the Cantor part}
$\mathrm{E}^c v$ is singular with respect to $\mathcal{L}^n$ and such that $|\mathrm{E}^c v|(B)=0$ if $\hn(B)<\infty$, while $\mathrm{E}^j v$ is concentrated on $J_v$. The density of $\mathrm{E}^a v$ with respect to $\mathcal{L}^n$ is denoted by $e(v)$, and we have that (see \cite[Theorem~4.3]{AmbCosDM97}) for $\mathcal{L}^n$-a.e.\ $x\in U$
\begin{equation*}\label{3105171931}
\lim_{\varrho \to 0^+} \int \limits_{B_\varrho(x)} \frac{\big|\big(v(y)-v(x)-e(v)(x)(y-x)\big)\cdot (y-x)\big|}{|y-x|^2}\, \mathrm{d}y=0\,.
\end{equation*}
The space $SBD(U)$ is the subspace of all functions $v\in BD(U)$ such that $\mathrm{E}^c v=0$, while for $p\in (1,\infty)$
\begin{equation*}
SBD^p(U):=\{v\in SBD(U)\colon e(v)\in L^p(\Omega;\Mnn),\, \hn(J_v)<\infty\}\,.
\end{equation*}
Analogous properties hold for $BV$, as the countable rectifiability of the jump set and the decomposition of $\mathrm{D}v$, and the spaces $SBV(U;\R^m)$ and $SBV^p(U;\R^m)$ are defined similarly, with $\nabla v$, the density of $\mathrm{D}^a v$, in place of $e(v)$.
For more details on the spaces $BV$, $SBV$ and $BD$, $SBD$ functions, we refer to \cite{AFP} and to \cite{AmbCosDM97, BelCosDM98, Bab15, Tem}, respectively.

We briefly recall the definition and the main properties of $GSBD$ functions from \cite{DM13}, referring to that paper for a general treatment.
\begin{definition}\label{def:GBD}
A $\mathcal{L}^n$-measurable function $v\colon U\to \Rn$ is in $GBD(U)$ if there exists $\lambda_v\in \mathcal{M}^+_b(U)$ such that for every $\xi \in \Sn$ and $\tau \in C^1(\R)$ with $-\tfrac{1}{2}\leq \tau \leq \tfrac{1}{2}$ and $0\leq \tau'\leq 1$, we have 
\[
\mathrm{D}_\xi\big(\tau(v\cdot \xi)\big)=\mathrm{D}\big(\tau(v\cdot \xi)\big)\cdot \xi \in \mathcal{M}_b(U)\,,
\]
and 
\begin{equation*}
\big|\mathrm{D}_\xi\big(\tau(v\cdot \xi)\big)\big|(B)\leq \lambda_v(B) \quad\text{ for $B\subset U$ Borel};
\end{equation*}
The function $v$ belongs to $GSBD(U)$ if $v\in GBD(U)$ and 
\[
U^\xi_y \ni t \mapsto \widehat{v}^\xi_y(t):=v(y+t \xi) \cdot \xi \in SBV_{\mathrm{loc}}(U^\xi_y)
\]
 for every $\xi \in \Sn$ and for $\hn$-a.e.\ $y\in \Pi^\xi$, where $U^\xi_y:= \{ t\in \R \colon y + t \xi \in U\}$  and $\Pi_\xi:=\{ y \in \Rn \colon y \cdot \xi = 0\}$. 
\end{definition}
For every $v\in GBD(U)$ the \emph{approximate jump set} $J_v$ is still countably $(\hn,n-1)$-rectifiable (\textit{cf.}~\cite[Theorem~6.2]{DM13}) and $v$ has an \emph{approximate symmetric gradient} $e(v)\in L^1(U;\Mnn)$, characterised by \
\begin{equation*}
\lim_{\varrho \to 0^+} \int \limits_{B_\varrho(x)} \psi\bigg(\frac{\big(v(y)-v(x)-e(v)(x)(y-x)\big)\cdot (y-x)}{|y-x|^2} \bigg)\, \mathrm{d}y=0\,
\end{equation*}
for $\psi$ a homeomorphism between $\Rn$ and a bounded subset of $\Rn$. If $v \in GSBD(U)$, with $e(v) \in L^p(U; \Mnn)$, $p>1$, and $\hn(J_v) < \infty$, then $v \in GSBD^p(U)$.

The following result has been proven by Chambolle, Conti, and Francfort in \cite{CCF16}, stated in $SBD^p$. The proof, only based on one dimensional slicing, holds in fact for functions in $GSBD^p$, and this has been employed for instance in  \cite{ChaConIur17, CC17}. 
\begin{proposition}\label{prop:3.1CCI}
Let $0<\theta''<\theta' <1$, $Q =(-r,r)^n$, $Q'=(-r\theta', r\theta')^n$, $u\in GSBD^p(Q)$, $p\in [1,\infty)$. Then there exist a Borel set $\omega\subset Q'$ and an affine function $a\colon \Rn\to\Rn$ with $e(a)=0$ such that $\mathcal{L}^n(\omega)\leq c_* r \hn(J_u)$ and
\begin{equation}\label{prop3iCCF16}
\int\limits_{Q'\setminus \omega}(|u-a|^{p}) ^{1^*} \dx\leq c_* r^{(p-1)1^*}\Bigg(\int\limits_Q|e(u)|^p\dx\Bigg)^{1^*}\,,
\end{equation}
 with $1^*:= \frac{n}{n-1}$. 
If additionally $p>1$, then there is $q>0$ (depending on $p$ and $n$) such that, for a given mollifier $\varrho_r\in C_c^{\infty}(B_{(\theta'-\theta'')r})\,, \varrho_r(x)=r^{-n}\varrho_1(x/r)$, the function $v=u \chi_{Q'\setminus \omega}+a\chi_\omega$ obeys
\begin{equation}\label{prop3iiCCF16}
\int\limits_{Q''}|e(v\ast \varrho_r)-e(u)\ast \varrho_r|^p\dx\leq c_*\left(\frac{\hn(J_u)}{r^{n-1}}\right)^q \int\limits_Q|e(u)|^p\dx\,,
\end{equation}
where $Q''=(-r \theta'', r \theta'')^n$.
The constant in  \eqref{prop3iCCF16}  depends only on $p$, $n$, and $\theta'$, the one in  \eqref{prop3iiCCF16}  also on $\varrho_1$ and $\theta''$.  
\end{proposition}

\begin{remark}\label{0406181929}
 By H\"older inequality and \eqref{prop3iCCF16} it follows that
 \begin{equation}\label{prop3CCFHolder}
\int\limits_{Q'\setminus \omega} |u-a|^{p} \dx\leq \mathcal{L}^n(Q'\setminus \omega)^{1/n}\Bigg(\int\limits_{Q'\setminus \omega}(|u-a|^{p}) ^{1^*} \dx\Bigg)^{1/1^*}\hspace{-1em}\leq  c r^p \int\limits_Q|e(u)|^p\dx
 \end{equation}
 \end{remark}
 \par
\medskip
\paragraph{\bf Some regularity results.}
 For every $\gamma\geq 0$ let
\begin{equation}\label{1006181950}
H_\gamma:=\{ x=(x', x_n) \in \Rn \colon x_n > -\gamma\}\,,
\end{equation}
and  (recall the notation $B_\varrho$ for $B_\varrho(0)$) 
\begin{equation}\label{0906181125}
E_{0,\gamma}(u, B_\varrho):=\begin{dcases}
\int \limits_{B_\varrho} \C e(u) \colon e(u) \dx \quad\text{if }u\in H^1(B_\varrho;\Rn),\, u=0 &\text{a.e.\ in }B_\varrho \sm H_\gamma\,,\\
 +\infty \quad &\text{otherwise.}
\end{dcases}
\end{equation}

\begin{definition}\label{def:locmin}
We say that $u\in H^1(B_\varrho;\Rn)$ is a local minimiser of $E_{0,\gamma}(\cdot, B_\varrho)$ if 
\begin{equation*}
E_{0,\gamma}(u, B_\varrho) \leq E_{0,\gamma}(v, B_\varrho)
\end{equation*}
for every $v \in H^1(B_\varrho;\Rn)$ with $\{u\neq v\}\Subset B_\varrho$.
\end{definition}
We now consider two regularity results for minimisers of $E_{0,\gamma}$, that solve in a weak form the elliptic equation $\mathrm{div\,}\C e(u)=0$ in $B_\varrho \cap H_\gamma$. 
These are useful to prove the decay estimate in Lemma~\ref{le:6.6BG}.
The first one follows from the fact that the solutions of $\mathrm{div\,}\C e(u)=~0$ are expressed through a $(2{-}n)$-essentially homogeneous $\C$-dependent kernel (see \cite[Theorem~6.2.1, paragraph 6.2]{Mor66} and \cite[Section~5]{ChaConIur17}). The second one concerns the boundary estimates and its proof, in Appendix, follows the lines of \cite[Theorem~4.18, (i)]{McLean}.
 \begin{theorem}\label{teo:ellreg1}
 Let $\gamma>1/2$ and $u\in H^1(B_{1/2};\Rn)$ be a local minimiser of $E_{0,\gamma}(\cdot, B_{1/2})$. Then there exists $C_0>0$, depending only on $\C$  and $n$,  such that
 \begin{equation*}
 \int \limits_{B_{\varrho/2}} \C e(u)\colon e(u) \dx \leq C_0 \, \varrho^n \int \limits_{B_{1/2}} \C e(u) \colon e(u)\dx\,,
 \end{equation*}
 for every $\varrho\leq 1/2$. 
 \end{theorem}

 \begin{theorem}\label{teo:ellreg2}
 Let $\gamma\in [0, 1/2]$, $u\in H^1(B_1;\Rn)$ be a local minimiser of $E_{0,\gamma}(\cdot, B_1)$, and $R_0<1$ be such that $\frac{3}{4}R_0 > \gamma$. Then there exists $C'_0>0$, depending only on $\C$, $R_0$,  and $n$,  such that
 \begin{equation*}
 \int \limits_{B_{\varrho}} \C e(u)\colon e(u) \dx \leq C'_0 \, \varrho^n \int \limits_{B_{R_0}} \C e(u) \colon e(u)\dx\,,
 \end{equation*}
 for every $\varrho\leq \frac{3}{4} R_0$.
 \end{theorem}

\section{Approximation of functions with small jump \\ prescribing a value in a subdomain}

In this section we approximate $GSBD^p$ functions with jump set small in $\hn$-measure, by functions keeping the same value near the boundary and smooth in the interior. The different point with respect to \cite[Theorem~3]{ChaConIur17}  is that we also want the approximation to have the same value (here 0) on a subdomain $D$. Then we modify the construction in \cite[Theorem~3]{ChaConIur17} in the interior part, where the original function is regularised, keeping 0 in a neighbourhood of $D$. As in \cite{ChaConIur17} this is done in a cubic domain for simplicity of notation, but holds also for balls (see Remark~\ref{0406181215}). We consider here a bulk energy positively $p$-homogeneous in $e(u)$.

Let $Q_r:=(-r,r)^n$ and $Q:=Q_1=(-1,1)^n$. Moreover, let \[f_0(\xi):=\frac{1}{p}\big(\C \xi \cdot \xi\big)^{p/2} \quad \text{ for $\xi \in \Mnn$}
\]
with $\C$ satisfying \eqref{eq:hpC}, and let $\varrho\in C_c^\infty(B(0,1/6);\R^+)$ be a radially symmetric mollifier, with $\varrho_\delta(x):=\delta^{-n} \varrho(\delta^{-1} x)$ for every $\delta>0$. In the following, for a given subset $U$ of $Q$, we denote $U_\delta:=Q \cap \big( U+(-3\delta, 3\delta)^n \big)$.
\begin{theorem}\label{teo:teor3CCI}
Let $D:= \{ x=(x',x_n)\in Q \colon x_n \leq  g(x')\} \subset Q$, with $g \in C^1(\R^{n-1})$ and $\mathrm{Lip}(g)\leq \frac{1}{8}$ \footnote{later we consider $g_h$ with $\mathrm{Lip}(g_h)$ vanishing}. Then there exist $\eta \in (0,1)$, $C>1$ depending only on $n$, $p$, $\C$, such that for every $u\in GSBD^p(Q)$ with $u=0$ in $D$ and 
\[
\delta:=\hn(J_u)^{1/n} < \eta\,,
\]
there are $R \in (1-\sqrt{\delta},1)$, $\tilde{u}\in GSBD^p(Q)$ with $\tilde{u}=0$ in $D$, and $\tilde{\omega}\subset Q_R$ such that
\begin{itemize}
\item[1.] $\tilde{u}\in C^\infty(Q_{1-\sqrt{\delta}}),\, \tilde{u}=u$ in $Q\sm Q_R$, $\hn(J_u\cap \partial Q_R)=\hn(J_{\tilde{u}}\cap \partial Q_R)=0$;
\item[2.] $\hn(J_{\tilde{u}} \sm J_u)\leq C \sqrt{\delta}\hn(J_u \cap (Q\sm Q_{1-\sqrt{\delta}}))$;
\item[3.] There is $s\in (0,1)$ depending only on $n$ and $p$ such that
\begin{equation*}
\|e(\tilde{u}) - \varrho_\delta \ast e(u)\|_{L^p(Q_{1-\sqrt{\delta}};\Mnn)} \leq C \delta^s \|e(u)\|_{L^p(Q;\Mnn)} + C \|e(u)\|_{L^p(D_\delta;\Mnn)}
\end{equation*}
and
\begin{equation}\label{0506182341}
\int \limits_U f_0(e(\tilde{u})) \dx \leq C \int \limits_{U_\delta} f_0(e(u))\dx + C \delta^s \int \limits_Q f_0(e(u)) \dx  
\end{equation}
for any $U\Subset Q$;
\item[4.] $|\tilde{\omega}| \leq C \delta \big(1+ \hn(J_u \cap Q_R)\big) $ and 
\begin{equation}\label{eq:pr4}
\int \limits_{Q\sm \tilde{\omega}}|\tilde{u}-u|^p \dx \leq C \delta^p \int \limits_Q |e(u)|^p \dx\,;
\end{equation}
\item[5.] If $\psi \in \mathrm{Lip}(Q;[0,1])$, then, for $s\in (0,1)$ as in 3.,
\begin{equation}\label{0506182359}
\int \limits_Q\psi f_0(e(\tilde{u})) \dx \leq C \int \limits_Q \psi f_0(e(u)) \dx + C \delta^s\big(1+\mathrm{Lip}(\psi) \big) \int \limits_Q |e(u)|^p\dx \,;
\end{equation}
\item[6.] If $u\in L^p(\Omega;\Rn)$, then for $U\Subset Q$
\begin{equation*}
\|\tilde{u}\|_{L^p(U;\Rn)}\leq \|u\|_{L^p(U;\Rn)} + C \delta^\frac{1}{2p} \big( \|u\|_{L^p(Q;\Rn)} + \|e(u)\|_{L^p(Q;\Rn)} \big)\,.
\end{equation*}
\end{itemize}
\end{theorem}

\begin{remark}\label{0406181215}
The properties 1., 2., 6.\  of Theorem~\ref{teo:teor3CCI} are in common  with the corresponding ones of \cite[Theorem~3]{ChaConIur17} (we have stated Theorem~\ref{teo:teor3CCI} similarly to  \cite[Theorem~3]{ChaConIur17} for an easier comparison), while in 3., 4., 5.\ we have a further contribution localised in $D_\delta \sm D$ that introduces the factor $C>1$ in place of 1 in \eqref{0506182341} and \eqref{0506182359}  (notice that $u=0$ in $D$ so we could replace $D_\delta$ by $D_\delta \sm D$ in 3., while in 4.\ we used  $|(\partial D)_\delta| \leq C\delta$): this is due to the fact that we have to correct the modified function $\tilde{u}$ with respect to the one in  \cite[Theorem~3]{ChaConIur17}, in order to guarantee that $\tilde{u}=0$ in $D$.
\end{remark}

\begin{proof}
We follow the notation in \cite[Theorem~3]{ChaConIur17}, but modifying the approximating function $\tilde{u}$, so we recall the first part of the construction therein. In the following $C$ will indicate a generic constant depending on $n$, $p$, $\C$.

Denoting $N:=[1/\delta]$,  so that $({-}N\delta, N\delta)^n \subset Q$,  let for $i=0,\dots, N-1$
\begin{equation*}
Q^i:=\big( {-}(N{-}i)\delta, (N{-}i) \delta \big)^n\,,\qquad C^i:=Q^i\sm Q^{i+1}\,,
\end{equation*}
with  $C^{N-1}=Q^{N-1}$. 
Up to a small translation, one may assume that for every $i=0,\dots, N-1$
\begin{subequations}
\begin{align}
\hn(J_u\cap \partial Q^i)&=0\,,\label{0906180755}\\
\lim_{r\to 0} r^{-n} \int_{B_r(y)} |e(u)-e(u)(y)|^p \dx&=0 \quad  \text{ for $\hn$-a.e.\ $y$ in } \partial Q^i\,.
\end{align}
\end{subequations}
Moreover, one can find (see \cite[(10)]{ChaConIur17}) $i_0 \in \N \cap [1,1/\sqrt{\delta}-3]$ such that, for $\delta$ small enough,
\begin{align*}
\int \limits_{C^{i_0} \cup C^{i_0+1}} |e(u)|^p \dx &\leq 8\sqrt{\delta}\int\limits_{Q\sm Q_{1-\sqrt{\delta}}} |e(u)|^p \dx\,,\\
\hn\big(J_u\cap (C^{i_0} \cup C^{i_0+1})\big)&\leq 8 \sqrt{\delta} \hn\big(J_u \cap (Q\sm Q_{1-\sqrt{\delta}})\big)\,.
\end{align*}
In particular $Q_{1-\sqrt{\delta}}\subset  Q^{i_0+1}$. 

The cube $Q^{i_0+1}$ is divided into cubes $z+(0,\delta)^n$, $z\in \delta \Z^n$, the crown $C^{i_0}$ into dyadic slabs
\begin{equation*}
S_k:=\big( -(N-i_0-2^{-k})\delta, (N-i_0-2^{-k})\delta \big)^n \sm \big( -(N-i_0-2^{-k+1})\delta, (N-i_0-2^{-k+1})\delta \big)^n\,,
\end{equation*}
and each $S_k$ into cubes of the type $z+(0, \delta 2^{-k})^n$, $z\in 2^{-k} \delta \Z^n$. 

The set of all the cubes introduced is called $\mathcal{W}$, and $\mathcal{W}_0$ is the set of cubes covering $Q^{i_0+1}$. For any $q\in \mathcal{W}$, let $q'$, $q''$, $q'''$ be the cubes with same center and dilated by a factor $7/6$, $4/3$, $3/2$, respectively. For
\begin{equation*}
\eta:= \frac{1}{2\cdot 8^n c_*}\,
\end{equation*}
where $c_*$ is the constant in Proposition~\ref{prop:3.1CCI} obtained for $\theta'=8/9$ (the sidelength ratio between $q''$ and $q'''$),
a cube $q\in \mathcal{W}$ of sidelength $\delta_q$ is called \emph{good} if
\begin{equation*}
\hn(q'''\cap J_u)\leq \eta \, (\delta_q)^{n-1}\,,
\end{equation*}
otherwise it is \emph{bad}. By the assumption $\delta<\eta$ and since $\delta^n=\hn(J_u)$, then each $q\in \mathcal{W}_0$ is good.
The union of bad cubes is denoted by $\mathcal{B}$ and (see \cite[(12)]{ChaConIur17})
\begin{align*}
\hn(\partial \mathcal{B}) &\leq \frac{C}{\eta} \sqrt{\delta}\,\hn\big(J_u\cap (Q\sm Q_{1-\sqrt{\delta}})\big)\,,\\
|\mathcal{B}|&\leq \frac{C}{\eta} \delta^{3/2} \,\hn\big(J_u\cap (Q\sm Q_{1-\sqrt{\delta}})\big)\,.
\end{align*}
The good cubes are enumerated into $(q_i)_{i=1}^\infty$, such that $\mathcal{W}_0=\bigcup_{i\leq N_0} q_i$ for $N_0=2^n(N{-}i_0{-}1)^n$. Consider a partition of unity  $(\varphi_i)_{i=1}^\infty$ of $Q^{i_0} \sm \overline{\mathcal{B}}$, in correspondence to $(q_i)_{i=1}^\infty$, with  $\varphi_i=0$ on $Q^{i_0} \sm q'_i \sm \mathcal{B}$.  
Then $\sum_i \varphi_i=1$ on  $Q^{i_0} \sm \overline{\mathcal{B}}$ (a locally finite sum),
\begin{subequations}
\begin{equation}\label{0506181828}
\varphi_i \text{ is } C^\infty \text{ in } Q^{i_0} \sm  \overline{\mathcal{B}}\,,  \qquad  \overline{ \{\varphi_i \neq 0\} } \subset q'_i  \,, 
\end{equation}
and we may ensure that
\begin{equation}\label{0710181303}
|\nabla \varphi_i|\leq \frac{C}{\delta_{q_i}}\,.
\end{equation}
\end{subequations}
Moreover, by Proposition~\ref{prop:3.1CCI} (and Remark~\ref{0406181929}), for any $q_i$ there is a set $\omega_i\subset q''_i$ with 
\[
|\omega_i|\leq c_* \delta_{q_i} \hn(J_u\cap q'''_i)\leq c_* \eta\, (\delta_{q_i})^n\]
and $a_i$ affine with $e(a_i)=0$ such that
\begin{subequations}
\begin{align}
\|u-a_i\|_{L^p(q''_i \sm \omega_i)} &\leq c_* \delta_{q_i} \|e(u)\|_{L^p(q'''_i)} \,, \label{0406181937}\\
\| e(u_i)-e(u)\ast \varrho_{\delta_{q_i}} \|_{L^p(q'_i)} &\leq  c_*  \bigg(\frac{\hn(J_u \cap q'''_i)}{(\delta_{q_i})^{n-1}} \bigg)^{\ol p} \|e(u)\|_{L^p(q'''_i)}\,, \label{0406181943}   
\end{align}
\end{subequations}
where $\ol p=q/p$ ($q$ as in Proposition~\ref{prop:3.1CCI}) and
\begin{equation}\label{0406181944}
u_i:=\varrho_{\delta_{q_i}} \ast (u\chi_{q''_i \sm \omega_i} + a_i \chi_{\omega_i})\,.
\end{equation}
We are now ready to define our approximating function, by
\begin{equation}\label{0406182012}
\tilde{u}:=\begin{cases}
\sum \limits_{q'_i \cap D = \emptyset} \varphi_i \, u_i & \quad \text{in }Q^{i_0}\sm \mathcal{B}\,, \\
\hspace{2em} u &\quad \text{in }\mathcal{B}\cup (Q\sm Q^{i_0})\,.
\end{cases}
\end{equation}
Notice that $\tilde{u}=0$ in $D$, since $\mathrm{supp}(\varphi_i) \cap D \subset q'_i \cap D=\emptyset$ for any $i$ in the sum, and since $u=0$ on $D$. Since  $Q_{1-\sqrt{\delta}}\subset  Q^{i_0+1}$  and  $\tilde{u}$ is smooth in  $Q^{i_0}  \sm \overline{\mathcal{B}}$,  taking $Q_R=Q^{i_0}$ (recall also \eqref{0906180755}) we get  the first two assertions of  property 1  (we see below that $\hn(J_u\cap \partial Q^{i_0})=\hn(J_{\tilde{u}}\cap \partial Q^{i_0})=0$). 
We observe also that the approximating function  in \cite[Theorem~3]{ChaConIur17}, that we denote $\widehat{u}$, is defined as $\sum_i \varphi_i \, u_i$ in $Q^{i_0}\sm \mathcal{B}$ and $u$ elsewhere, so that
\begin{equation}\label{0406182052}
v:=\widehat{u} - \tilde{u} = \sum \limits_{q_i' \cap D \neq \emptyset} \varphi_i \, u_i\,.
\end{equation}
Since  $\varphi_i$ are  smooth  and we see below that $\hn(J_u\cap \partial Q^{i_0})=\hn(J_{\tilde{u}}\cap \partial Q^{i_0})=0$,  property 2. follows directly from the analogue of \cite[Theorem~3]{ChaConIur17}. We have that
\begin{equation}\label{0506181826}
e(v)=\sum \limits_{q_i' \cap D \neq \emptyset} (\varphi_i \, e(u_i) + \nabla \varphi_i \odot u_i )\,.
\end{equation}
We now estimate $\|u_i\|_{L^p(q'_i)}$ for every $q_i$ with $q'_i \cap D \neq \emptyset$. For these cubes, since $\mathrm{Lip}(g)<\frac{1}{8}$, we have that
\begin{equation}\label{0406182108}
\frac{|(q''_i \sm \omega_i) \cap D|}{|q''_i|}\geq d_0 >0\,,
\end{equation}
with $d_0$ a dimensional constant, so independent of $i$. Since $u=0$ in $D$, \eqref{0406181937} gives
\begin{equation*}
\|a_i\|_{L^p((q''_i \sm \omega_i)\cap D)} \leq c_* \delta_{q_i} \|e(u)\|_{L^p(q'''_i)}\,,
\end{equation*}
and since $a_i$ is affine (see e.g.\ \cite[Lemma~3.4]{ChaConIur17}), by \eqref{0406182108} we deduce
\begin{equation}\label{0406182109}
\|a_i\|_{L^p(q''_i)} \leq C \delta_{q_i} \|e(u)\|_{L^p(q'''_i)}\,,
\end{equation}
with $C$ depending on $d_0$ and $c_*$. Moreover, employing the fact that (since $\varrho_{\delta_{q_i}} \ast a_i=a_i$ because $\varrho_{\delta_{q_i}}$ is radial) 
\begin{equation}\label{0506181835}
u_i - a_i= \varrho_{\delta_{q_i}} \ast \big( (u-a_i) \chi_{q''_i \sm \omega_i} \big)\,,
\end{equation} 
we get that
\begin{equation}\label{0406182119}
\int \limits_{q'_i} |u_i - a_i|^p\dx \leq \int \limits_{q''_i \sm \omega_i} |u-a_i|^p \dx \leq c_* (\delta_{q_i})^p \int \limits_{q'''_i} |e(u)|^p \dx\,.
\end{equation}
Notice that this holds for any good cube.
Collecting \eqref{0406182109} and \eqref{0406182119} we get
\begin{equation}\label{0406182120}
\|u_i\|_{L^p(q'_i)} \leq C \delta_{q_i}  \|e(u)\|_{L^p(q'''_i)} \,,
\end{equation}
for every $q_i$ with $q'_i \cap D \neq \emptyset$.
Setting $\widehat{\omega}:= \bigcup_{q'_i \cap D = \emptyset} \omega_i \sm \mathcal{B}$ and $\delta_k := \delta 2^{-k}$, we get that for any $y \in \partial Q^{i_0}$ (recall \eqref{0406182052})
\begin{equation*}
\int\limits_{B_{\delta_k}(y) \sm \widehat{\omega}} \hspace{-1em} |v|^p \dx \leq C \delta^p \, 2^{-kp} \hspace{-1em}\int \limits_{B_{\delta_{k-1}}(y)}\hspace{-1em} |e(u)|^p \dx\,,
\end{equation*}
and $|\widehat{\omega} \cap B_{\delta_k}(y)| \leq C \delta_k (1+1/\eta) \hn(J_u \cap B_{\delta_{k-1}}(y))$. Arguing as in \cite{ChaConIur17}, we deduce then that for $\hn$-a.e.\ $y\in \partial Q^{i_0}$
\begin{equation*}
\frac{1}{\delta_k^n}|\{ |v| > \varepsilon \} \cap B_{\delta_k}(y)| \leq \frac{c_* \hn(J_u \cap B_{\delta_k}(y))}{\delta_k^{n-1}} + C_{\varepsilon, c_*}\, \delta_k \Bigg( \frac{1}{\delta_{k-1}^n}\int\limits_{B_{\delta_{k-1}}(y)} \hspace{-1em} |e(u)|^p \dx \Bigg)^{1/p}\,,
\end{equation*}
and the right hand side vanishes as $\delta_k \to 0$, by \eqref{0906180755}; then $v$ has null trace on $\partial Q^{i_0}$, $\tilde{u} \in GSBD^p(Q)$, as well as $\widehat{u}$, and $\hn(J_{\tilde{u}}\cap \partial Q^{i_0})=\hn(J_{\widehat{u}}\cap \partial Q^{i_0})=\hn(J_u\cap \partial Q^{i_0})=0$ (the properties for $\widehat{u}$ follows from \cite[Theorem~3]{ChaConIur17}). This concludes also properties~1., 2.

 By \eqref{0710181303} and \eqref{0406182120} it follows that
\begin{equation}\label{0506181829}
\|\nabla \varphi_i \odot u_i\|_{L^p(q'_i)} \leq C  \|e(u)\|_{L^p(q'''_i)} \,.
\end{equation}
Moreover, since $\varphi \in [0,1]$, $\varrho_{\delta_{q_i}} \ast a_i =a_i$, $e(a_i)=0$, \eqref{0506181835} and \eqref{0406182119} imply that
\begin{equation*}\label{0506181832}
\begin{split}
\|\varphi_i \, e(u_i)\|_{L^p(q'_i)} &\leq \|e(u_i)\|_{L^p(q'_i)}= \|e(u_i - a_i)\|_{L^p(q'_i)} = \|\varrho_{\delta_{q_i}} \ast \big( (u-a_i) \chi_{q''_i \sm \omega_i} \big)\|_{L^p(q'_i)} \\& \leq \frac{\|\nabla \varrho\|_1}{\delta_{q_i}} \|u-a_i\|_{L^p(q''_i \sm \omega_i)} \leq C  \|e(u)\|_{L^p(q'''_i)} \,.
\end{split}
\end{equation*} 
Putting together the above inequality with \eqref{0506181829} we get
\begin{equation}\label{0506181841}
\|\varphi_i \, e(u_i) + \nabla \varphi_i \odot u_i \|_{L^p(q'_i)} \leq C  \|e(u)\|_{L^p(q'''_i)}
\end{equation}
Therefore, by \eqref{0506181826}, since the $q'''_i$ are finitely overlapping and since, if $q'_i \cap D \neq \emptyset$, then $q'''_i \subset D_\delta$,
\begin{equation}\label{0506181843}
\|e(v)\|_{L^p(U)} \leq C \|e(u)\|_{L^p((D\cap U)_\delta)}\,,
\end{equation}
in particular \[\|e(v)\|_{L^p(Q)} \leq C \|e(u)\|_{L^p(D_\delta)}\,.\]
Since in \cite[Theorem~3, property 3.]{ChaConIur17} it is proven that
\begin{equation*}
\|e(\widehat{u}) - \varrho_\delta \ast e(u)\|_{L^p(Q_{1-\sqrt{\delta}})} \leq C \delta^s \|e(u)\|_{L^p(Q)} 
\end{equation*}
and
\begin{equation*}
\bigg(\int \limits_U f_0(e(\widehat{u})) \dx\bigg)^{\frac{1}{p}} \leq \bigg(\int \limits_{U_\delta} f_0(e(u))\dx\bigg)^{\frac{1}{p}} + C \delta^s \bigg(\int \limits_Q f_0(e(u)) \dx\bigg)^{\frac{1}{p}}\,, 
\end{equation*}
then \eqref{0506181843} gives the first part of property 3.\ and
\begin{equation*}
\bigg(\int \limits_U f_0(e(\tilde{u})) \dx\bigg)^{\frac{1}{p}} \leq \bigg(\int \limits_{U_\delta} f_0(e(u))\dx\bigg)^{\frac{1}{p}} + C \delta^s \bigg(\int \limits_Q f_0(e(u)) \dx\bigg)^{\frac{1}{p}} + C \bigg( \int \limits_{(D\cap U)_\delta} f_0(e(u)) \dx \bigg)^{\frac{1}{p}}\,,
\end{equation*}
from which also \eqref{0506182341} follows, using $(a+b)^p \leq 2^{p-1} (a^p+b^p)$.

As for 4., this follows from the fact that $\mathrm{supp }\, v \subset (\partial D)_\delta$, whose volume is less than $C \delta$, and since property 4.\ of 
\cite[Theorem~3]{ChaConIur17} gives \eqref{eq:pr4} for $\widehat{u}$ for a set $\widehat{\omega}$ with $|\widehat{\omega}| \leq C \delta  \hn(J_u \cap Q_R)$.

Let us now consider property 5., so fix $\psi \in \mathrm{Lip}(Q; [0,1])$: we have
\begin{equation*}
\int \limits_Q \psi \, f_0(e(\tilde{u}))\dx = \int \limits_Q \int_0^{\psi(x)} f_0(e(\tilde{u})) \, \mathrm{d}t \dx = \int_0^1 \int\limits_{\{ x \colon t < \psi(x)\} } f_0(e(\tilde{u})) \dx \, \mathrm{d}t\,.
\end{equation*}
Taking $U=\{ x \colon t <\psi(x) \}$ in \eqref{0506182341} (notice that $U_\delta \subset \{x \colon t <\psi(x) + c_\psi \delta \}$, where $c_\psi=3n^{1/2}\mathrm{Lip}(\psi)$) we get
\begin{equation}
\begin{split}
\int \limits_Q \psi \, f_0(e(\tilde{u}))\dx &\leq \int_0^1 C \,\Big( \hspace{-1.5em} \int \limits_{ \{x \colon t <\psi(x) + c_\psi \delta \} } \hspace{-1.5em}  f_0(e(u)) \dx +  \delta^s \int \limits_Q f_0(e(u)) \dx \Big) \mathrm{d}t \\& = C \int \limits_Q  (\psi(x) + c_\psi \delta) f_0(e(u)) \dx + C \delta^s \int \limits_Q f_0(e(u)) \dx\,,
\end{split}
\end{equation}
which implies \eqref{0506182359}.

Property 6.\ follows since it holds for $\widehat{u}$ in place of $\tilde{u}$, and by \eqref{0406182052}, \eqref{0406182120} we have $\|v\|_{L^p(Q)} \leq C \delta \|e(u)\|_{L^p(Q)}$.
\end{proof}
\begin{remark}\label{rem:0906180843}
Theorem~\ref{teo:teor3CCI} and \cite[Theorem~3]{ChaConIur17} hold for any cube $Q_t$ in place of $Q$, arguing in the same way, and also for any other regular open sets, as balls, employing a Whitney-type argument (see also the comments at the beginning of Subsection~3.2 in \cite{ChaConIur17}).
\end{remark}
\section{Minimising sequences for the Griffith energy with Dirichlet condition\\ and vanishing jump}
For every $D\subset Q$ Borel set, $u\in GSBD^p(Q)$, $c>0$,  and $A\subset Q$ open set, we define
\begin{equation}\label{0606181912}
G_D(u, c, A):=\begin{dcases}
\int\limits_A f_0(e(u)) \dx + c \hn(J_u \cap A) \quad &\text{if } u=0 \text{ a.e.\ in }D\,, \\
+\infty & \text{otherwise.}
\end{dcases}
\end{equation}
Let us also set
\begin{equation*}\label{0606181929}
\m_D(u, c, A):= \inf \{ G_D(v, c, A) \colon v\in GSBD^p(Q), \{v \neq u\} \Subset A \}\,,
\end{equation*}
as the local minimum with respect to perturbations in $A$, and the deviation from minimality on $A$ given by (for $\m_D(u,c, A) <\infty$)
\begin{equation*}
\mathrm{Dev}_D(u, c, A):= G_D(u, c, A) - \m_D(u, c, A)\,.
\end{equation*}
The following theorem, which is the goal of this section, proves the convergence of quasi-minimisers for $G_{D_h}$ with vanishing jump measure toward a minimiser of the bulk energy with respect to its own boundary value. It is a Dirichlet counterpart of \cite[Proposition~3.4]{CFI17DCL}, \cite[Theorem~4]{ChaConIur17}. 
\begin{theorem}\label{thm:Teor4}
Let 
$D_h:=\{x=(x',x_n) \in Q \colon x_n \leq g_h(x')\}$, for $g_h \colon \R^{n-1} \to \R$ continuous and converging locally uniformly to the constant function $-\gamma$, with $\gamma\in [0,1)$.  Let $v_h \in GSBD^p(Q)$, $c_h >0$ be such that
\begin{equation}\label{0606182345}
\begin{split}
\sup_{h\in \N} G_{D_h}(v_h, c_h, Q) &< \infty\,,\\
\lim_{h\to \infty} \mathrm{Dev}_{D_h}(v_h, c_h, Q) &= \lim_{h\to \infty} \hn(J_{v_h})=0\,. 
\end{split}
\end{equation}
Then there exists $v\in W^{1,p}(Q; \Rn)$ with $v=0$ in $Q \sm H_\gamma$, such that, up to a subsequence $h_j$, $v_{h_j}\to v$ a.e.\ in  $Q$ with
\begin{subequations}
\begin{equation}\label{0606182049}
\int \limits_Q f_0(e(v)) \dx \leq \int \limits_Q f_0(e(w))  \dx \quad \text{ for any } w\in W^{1,p}(Q;\Rn), \{w \neq v\}\Subset Q, \, w=0 \text{ in } Q\sm H_\gamma\,,
\end{equation}
and, for any $t\in (0,1)$,
\begin{align}
\lim_{j\to \infty} G_{D_{h_j}}(v_{h_j}, c_{h_j}, Q_t) &= \int \limits_{Q_t} f_0(e(v)) \dx \,; \label{0606182048}\\
e(v_{h_j}) \to e(v) \text{ in }L^p(Q_t;\Mnn), & \quad c_{h_j} \hn(J_{v_{h_j}} \cap Q_t)\to 0  \,. \label{0606182050}
\end{align}
\end{subequations}
\end{theorem}

\begin{proof}
Being $t \mapsto G_{D_h}(v_h, c_h, Q_t)$ nondecreasing in $[0,1]$, by Helly's theorem we have that up to  a subsequence (not relabelled) independent of $t$
\begin{equation}\label{0606182310}
\lim_{h\to \infty} G_{D_h}(v_h, c_h, Q_t)=: \Lambda(t) < \infty
\end{equation}
for every $t \in [0,1]$, and $\Lambda$ nondecreasing. By \cite[Proposition~2]{CCF16} there exist $\omega_h$ with $|\omega_h|\leq c \hn(J_{v_h})$ and $a_h$ affine with $e(a_h)=0$ such that
\begin{equation}\label{0606182318}
\int \limits_{Q\sm \omega_h} |v_h - a_h|^p \dx \leq C \int \limits_Q |e(v_h)|^p \dx\,.
\end{equation}
On the other hand we have that, since $g_h$ converge to $-\gamma > -1$ locally uniformly, then $|D_h| \geq d_0 >0$, which  by \cite[Lemma~3.4]{ChaConIur17}  (notice that $a_h$ are affine, $v_h=0$ in $D_h$, and $|\omega_h|\to 0$) gives 
\begin{equation}\label{0606182319}
\int \limits_Q |a_h|^p \leq C \int \limits_Q |e(v_h)|^p \dx\,,
\end{equation}
so that one can choose $a_h=0$ in \eqref{0606182318}. 
Being $e(v_h)$ bounded in $L^p$ we have that
\begin{equation}\label{0606182321}
v_h \,\chi_{Q\sm \omega_h} \weak v \quad\text{in }L^p(Q;\Rn)\,, 
\end{equation}
for a suitable $v\in L^p(Q;\Rn)$. Since $g_h$ converge to $-\gamma$ and $|\omega_h| \to 0$, by \eqref{0606182345}, we deduce that $v=0$ in $Q\sm H_\gamma$ and that $v_h$ converge pointwise to $v$ in $Q$, again up to a subsequence.  

Let us fix $\varepsilon>0$ and a point $t\in (0,1]$ of left continuity of $\Lambda$, so that we can find $t' \in (0,t)$ with 
\begin{equation}\label{0706180922}
\Lambda(t)-\Lambda(t') < \varepsilon\,,
\end{equation}
 and apply \cite[Theorem~3]{ChaConIur17} for $Q=Q_t$, $u=v_h$ (so $\delta_h= \hn(J_{v_h}\cap Q_t)^{1/n}$), which gives functions $\widehat{v}_h$ and exceptional sets $\widehat{\omega}_h$  (for their properties see also Theorem~\ref{teo:teor3CCI} and Remark~\ref{0406181215}):  in particular we have that $(\widehat{v}_h - v_h) \chi_{Q_t \sm \widehat{\omega}_h} \to 0$ in $L^p$ and $|\widehat{\omega}_h| \to 0$ as $h \to 0$, by \cite[Theorem~3, Property~4]{ChaConIur17}.   
  Moreover, the functions $\widehat{v}_h$ are $W^{1,p}(Q_s; \Rn)$ for any $s <t$ and $h$ large enough, and $e(\widehat{v}_h)$ is equibounded in $L^p$, by \cite[Theorem~3, Property~3]{ChaConIur17}: then Korn's inequality implies that $\widehat{v}_h - \widehat{a}_h$ converges weakly in $W^{1,p}_{\mathrm{loc}}(Q_t)$, for suitable affine $\widehat{a}_h$; actually this convergence holds true also with $\widehat{a}_h=0$, arguing as done for \eqref{0606182319}. 
 Recalling \eqref{0606182321}, we can say that  for any $\varphi \in C^\infty_c(Q_t)$
 \begin{equation*}
 \int\limits_{Q_t} \widehat{v}_h \cdot \varphi \dx  = \hspace{-1.5em} \int \limits_{Q_t \cap (\omega_h \cup \, \widehat{\omega}_h)} \hspace{-1.5em} \widehat{v}_h \cdot \varphi \dx + \int \limits_Q (v_h \, \chi_{Q \sm \omega_h}) \cdot (\varphi \, \chi_{Q_t \sm \widehat{\omega}_h}) \dx + \int \limits_{Q_t} \big( (\widehat{v}_h - v_h) \chi_{Q_t \sm \widehat{\omega}_h}  \big) \cdot \varphi\, \chi_{Q_t \sm \omega_h} \dx
 \end{equation*}
 converges to $\int \limits_Q v \cdot \varphi \dx$. We then deduce that
 
\begin{equation}\label{0606182334}
\widehat{v}_h \weak v \quad\text{in }W^{1,p}_{\mathrm{loc}}(Q_t)\,,
\end{equation}
that  (recall \cite[Theorem~3, Property~3]{ChaConIur17}) 
\begin{equation}\label{0606182336}
e(v_h) \weak e(v) \quad\text{in }L^p(Q_t;\Mnn)\,,
\end{equation}
and then that
\begin{equation}\label{0606182337}
\int \limits_{Q_t} f_0(e(v)) \dx \leq \liminf_{h\to \infty} \int \limits_{Q_t} f_0(e(v_h)) \dx \,.
\end{equation}
In particular, since $\int \limits_{Q_t} f_0(e(v_h)) \dx \leq G_{D_h}(v_h, c_h, Q_t)$, it follows that
\begin{equation}\label{0606182348}
\int \limits_{Q_t} f_0(e(v)) \dx \leq \Lambda(t) \,.
\end{equation}
We study now the local minimality of the limit function $v$, in the sense of \eqref{0606182049}, employing the quasi-local minimality of $v_h$ for $G_{D_h}(\cdot, c_h, Q)$. Then let us fix 
a test function for \eqref{0606182049}, that is 
$w\in W^{1,p}(Q;\Rn)$ with $\{w \neq v\} \Subset Q_t$ and $w=0$ a.e.\ in $Q \sm H_\gamma$. By \cite[Lemma~6.3]{BabGia14}, there exist $w_h\in W^{1,p}(Q;\Rn)$, with $w_h=0$ a.e.\ in $D_h$ and 
\begin{equation}\label{0706180845}
w_h \to w \quad\text{in }W^{1,p}(Q;\Rn)\,.
\end{equation} 
Since $\{w \neq v\} \Subset Q_t$, there is a $\widehat{t} \in (t', t)$ such that $Q_t \sm Q_{\widehat{t}} \subset \{w=v\}$.   Let $\psi \in C_c(Q_t)$ be a Lipschitz function  with $\psi=1$ in $\ol Q_{t'}$  and 
\begin{equation}\label{0906180850}
\{ 0 < \psi < 1\} \subset Q_{t''} \sm Q_{\widehat{t}} \subset \{w=v\} \cap Q_{t''} \sm   \ol Q_{t'} \,, \quad\text{for some }t''\in (t',t)\,.
\end{equation} 
Let us apply Theorem~\ref{teo:teor3CCI} for $Q=Q_t$, $D=D_h$, $u=v_h$, to get functions $\tilde{v}_h$ and exceptional sets $\tilde{\omega}_h$.
Notice that $\tilde{v}_h \in W^{1,p}(Q_{t''};\Rn)$ for $h$ large (so $\delta_h$ small, and $\tilde{v}_h\in C^\infty(Q_{t-\sqrt{\delta_h}})$) and that $e(\tilde{v}_h)$ is bounded in $Q_s$, for every $s<t$, by \eqref{0506182341}. Thus  we deduce, arguing as done before for the $\widehat{v}_h$, that 
\begin{equation}\label{0706180054}
\tilde{v}_h \weak v \quad\text{ in } W^{1,p}(Q_s;\Rn)\,,\text{ for }s<t\,,
\end{equation} 
by \eqref{0606182336} and Property 4.\ in Theorem~\ref{teo:teor3CCI} (observe that $|\tilde{\omega}_h|\leq C \delta_h \to 0$, and recall Korn's inequality).
We set
\begin{equation}\label{0706180017}
\tilde{w}_h:= \tilde{v}_h (1-\psi) + \psi \,w_h\,.
\end{equation}
Since $\tilde{v}_h=v_h$ in $Q\sm Q_{R_h}$, so $\{\tilde{v}_h \neq v_h\} \Subset Q_t$, and $\tilde{v}_h=w_h=0$ in $D_h$, we get
\begin{equation*}
\{\tilde{w}_h \neq v_h\}\Subset Q_t\,,\qquad \qquad \tilde{w}_h=0 \text{ in }D_h\,. 
\end{equation*} 
Therefore  $G_{D_h}(\tilde{w}_h, c_h, Q_t) < \infty$ and, by \eqref{0606182345}, 
\begin{equation}\label{0706180030}
\int \limits_{Q_t} f_0(e(v_h)) \dx + c_h \hn(J_{v_h} \cap Q_t) \leq \int \limits_{Q_t} f_0(e(\tilde{w}_h)) \dx + c_h \hn(J_{\tilde{w}_h} \cap Q_t) + o(1)\,,
\end{equation}
where $o(1)= \mathrm{Dev}_{D_h}(v_h, c_h, Q_t) \to 0$. By Properties 1.\ and 2.\ of Theorem~\ref{teo:teor3CCI} we have $J_{\tilde{w}_h} \subset J_{\tilde{v}_h} \subset Q_t \sm Q_{t-\sqrt{\delta_h}}$ and $\hn(J_{\tilde{v}_h}\sm J_{v_h})\leq C \sqrt{\delta_h}\hn(J_{v_h})$. This implies, subtracting $c_h \, \hn(J_{v_h} \cap Q_t \sm Q_{t-\sqrt{\delta_h}} )$ from both sides of \eqref{0706180030}, that
\begin{equation}\label{0706180041}
\int \limits_{Q_t} f_0(e(v_h)) \dx + c_h \hn(J_{v_h} \cap Q_{t-\sqrt{\delta_h}}) \leq \int \limits_{Q_t} f_0(e(\tilde{w}_h)) \dx+ o(1)\,.
\end{equation}
By \eqref{0706180017}
\begin{equation}\label{0706180841}
e(\tilde{w}_h)= (1-\psi) e(\tilde{v}_h) + \psi e(w_h) + \nabla \psi \odot (\tilde{v}_h-w_h)\,.
\end{equation}
In view of \eqref{0706180845}, \eqref{0906180850}, and \eqref{0706180054} we get
\begin{equation*}
\tilde{v}_h-w_h \to 0 \quad\text{in }L^p(\{0<\psi <1\};\Rn)\,,
\end{equation*}
and then, employing the convexity of $f_0$, 
\begin{equation}\label{0706180900}
\int \limits_{Q_t} f_0(e(\tilde{w}_h)) \dx \leq (1+o(1)) \bigg[  \int \limits_{Q_t} (1-\psi) f_0(e(\tilde{v}_h)) \dx + \int \limits_{Q_t} \psi f_0(e(w_h)) \dx \bigg] + o(1)\,.
\end{equation}
By Property~5.\ in Theorem~\ref{teo:teor3CCI}
\begin{equation}\label{0706180901}
\int \limits_{Q_t} (1-\psi) f_0(e(\tilde{v}_h)) \dx \leq C \int \limits_{Q_t} (1-\psi) f_0(e(v_h)) \dx + o(1)\,,
\end{equation}
so that, combining \eqref{0706180900}, \eqref{0706180901} with \eqref{0706180041},
\begin{equation}\label{0706180902}
\begin{split}
\int \limits_{Q_t} f_0(e(v_h)) \dx + c_h \hn(J_{v_h} \cap Q_{t-\sqrt{\delta_h}})  &\leq    \int \limits_{Q_t} \psi f_0(e(w_h)) \dx + o(1) \\& \hspace{1em}+ C \int \limits_{Q_t} (1-\psi) f_0(e(v_h)) \dx  \,.
\end{split}
\end{equation}
We now pass to the limit the above inequality employing \eqref{0606182336} and \eqref{0706180845}  respectively  in the left and in right hand side, obtaining
\begin{equation}\label{0706180933}
\int \limits_{Q_t}  f_0(e(v)) \dx \leq    \int \limits_{Q_t}  f_0(e(w)) \dx + C \limsup_{h\to \infty} \int \limits_{Q_t} (1-\psi) f_0(e(v_h)) \dx \,.
\end{equation}
Notice that, being $1-\psi=0$ in  $\ol Q_{t'}$  and $1-\psi \leq 1$ in $Q_t$,
\[
\int \limits_{Q_t} (1-\psi) f_0(e(v_h)) \dx \leq \int \limits_{Q_t \sm  \ol Q_{t'}} f_0(e(v_h)) \dx\,.
\]
Now, since 
\begin{equation*}
\begin{split}
\int \limits_{Q_t \sm Q_{t'}} f_0(e(v_h)) \dx&= G_{D_h}(v_h, c_h, Q_t)- G_{D_h}(v_h, c_h, Q_{t'}) - c_h \hn(J_{v_h} \cap Q_t \sm Q_{t'}) \\& \leq G_{D_h}(v_h, c_h, Q_t)- G_{D_h}(v_h, c_h, Q_{t'})\,,
\end{split}
\end{equation*}
we have that (recall \eqref{0706180922})
\begin{equation*}
\limsup_{h\to \infty} \int \limits_{Q_t \sm Q_{t'}} f_0(e(v_h)) \dx \leq \lim_{h\to \infty} \Big[ G_{D_h}(v_h, c_h, Q_t)- G_{D_h}(v_h, c_h, Q_{t'}) \Big]= \Lambda(t) -\Lambda(t') < \varepsilon\,.
\end{equation*}
Therefore from \eqref{0706180933} we deduce
\begin{equation}\label{0706180945}
\int \limits_{Q_t}  f_0(e(v)) \dx \leq    \int \limits_{Q_t} f_0(e(w)) \dx +C\,\varepsilon \,,
\end{equation}
and then \eqref{0606182049} follows by the arbitrariness of $\varepsilon$ and of the test function $w$.

 Moreover, we have that for $h$ large the left hand side of \eqref{0706180902} is greater than $G_{D_h}(v_h, c_h, Q_{t'})$, so that    
\begin{equation*}
 \Lambda(t) - \varepsilon < \Lambda(t')  \leq \int \limits_{Q_t} f_0(e(w)) \dx + C \varepsilon\,, 
\end{equation*}
for any $w$ test function for \eqref{0606182049}, and then
\begin{equation*}
\Lambda(t)\leq \int \limits_{Q_t} f_0(e(w)) \dx\,, 
\end{equation*}
since $\varepsilon$ is arbitrary. Taking $w=v$ and recalling \eqref{0606182348} we get
\begin{equation*}
\int \limits_{Q_t} f_0(e(v))\dx=\Lambda(t)
\end{equation*}
for every $t \in (0,1]$ point of left continuity of $\Lambda$. Since $t\mapsto \int \limits_{Q_t} f_0(e(v))$ is continuous, then it coincides for every $t$ with $\Lambda$ (that then is continuous too). By the definition \eqref{0606182310} of $\Lambda$ we conclude \eqref{0606182048}. At this stage, \eqref{0606182050} follows immediately from \eqref{0606182336} (that holds for every $t$) and \eqref{0606182048}. This completes the proof.
\end{proof}

\begin{remark}\label{rem:1606180911}
Employing the versions for balls of Theorem~\ref{teo:teor3CCI} and \cite[Theorem~3]{ChaConIur17} (see Remark~\ref{rem:0906180843}), we have that Theorem~\ref{thm:Teor4} holds also for balls $B_r$, $r>0$, in place of $Q$. In this version we apply it in the following section.
\end{remark}

\begin{remark}\label{rem:1206181208}
In Theorem~\ref{thm:Teor4}, if $p=2$ then \eqref{0606182049} corresponds to say that $v$ is a local minimiser of $E_{0,\gamma}(\cdot, Q)$ in the sense of Definition~\ref{def:locmin}.
\end{remark}
\section{Strong minimisers for the Griffith energy with Dirichlet condition}\label{Sec:strmin}
We assume, as in the Introduction, that $\Omega' \supset \Omega$ with $\Omega' \cap \dom=\dod$, $\mathrm{diam\,}\Omega'\leq 2\, \mathrm{diam\,}\Omega$, and introduce the following functional, defined for every open set $A\subset \Omega'$. Differently from the functional $G_D$ in \eqref{0606181912}, we consider the classical Griffith energy, so with the quadratic \emph{linearised elastic energy} as bulk energy.
 We then set for every $u\in GSBD^2(\Omega)$ and $A\subset \Omega'$
\begin{equation}\label{1006180918}
G_0(u, A):=\begin{dcases}
\int \limits_A \C e(u) \colon e(u) \dx + 2\beta \,\hn(J_u \cap A) &\quad\text{if }u=0 \text{ a.e.\ in } A \sm (\Omega\cup \dod)\,,\\
+\infty & \quad \text{otherwise.}
\end{dcases}
\end{equation}
Let also
\begin{equation}\label{1006180922}
\m(u,A):=\inf\{G_0(u,A) \colon v\in GSBD^2(A),\, \{v\neq u\} \Subset A\}
\end{equation}
be the local minimum value, and, if $\m(u,A)<+\infty$,
\begin{equation*}\label{1006180924}
\mathrm{Dev\,}(u,A):= G_0(u,A) - \m(u,A)
\end{equation*}
be the local deviation from minimality.
We state Theorem~\ref{teo:Teor3.4BG} for quasi-minimisers, which are defined as follows.
\begin{definition}\label{def:quasiMin}
A function $u\in GSBD^2(A)$ is a $(\omega,s)$-\emph{quasi-minimiser} of $G_0(\cdot,A)$ if there exist $\omega>0$ and $s\in (0,1)$ such that for every ball $B_\varrho(x) \subset A$  with $\varrho\leq 1$ 
\begin{equation*}
\mathrm{Dev\,}(u, B_\varrho(x)) \leq \omega \varrho^{n-1+s}\,.
\end{equation*}
\end{definition}

We are here interested in the Dirichlet minimisation problem
\begin{equation}\label{1006181103}
\min_{u\in GSBD^2(\Omega')} \left \{ \int \limits_{\Omega'} \C e(u) \colon e(u) \dx + 2\beta\, \hn(J_u \sm K) \colon u=u_0 \text{ in } \Omega' \sm (\Omega\cup \dod) \right\}\,,
\end{equation}
where $K\subset \Omega\cup \dod$ is closed in the relative topology.
In order to deal with the set $K$, we consider the following localised version of \eqref{1006181103}, still with Dirichlet boundary condition
\begin{equation}\label{1006181107}
\min_{u\in GSBD^2(A)} \left \{ \int \limits_{A} \C e(u) \colon e(u) \dx + 2\beta\, \hn(J_u \cap A) \colon u=u_0 \text{ in } A \sm (\Omega\cup \dod) \right\}\,,
\end{equation}
for every $A \subset \Omega'$.
The following proposition shows that there is a correspondence between solutions to \eqref{1006181107} and quasi-minimisers of $G_0(\cdot,A)$, for which the boundary condition is 0.  For the moment we do not assume that $\dod$ is of class $C^1$. 
\begin{proposition}\label{prop:1006181116}
Let $A\subset \Omega'$ open  such that $\hn(\dod \cap B_\varrho(x)) \leq \wt L \varrho^{n-1}$ for any $B_\varrho(x) \subset A$,  $u_0\in W^{1,\infty}(\Omega';\Rn)$, and $u\in GSBD^2(A)$ be a solution to \eqref{1006181107}. Then
\begin{equation*}
\widehat{u}:= u-u_0 \in GSBD^2(A)
\end{equation*}
is a $(\omega, 1/2)$-quasi-minimiser of $G_0(\cdot, A)$, with $\omega$ depending only on $n$, $\C$, $\|e(u_0)\|_\infty$, 
and  $\wt L$. 
\end{proposition}

\begin{proof}
Fix $B_\varrho(x) \subset A$  with $\varrho\leq 1$  and consider $v\in GSBD^2(A)$ with $G_0(v,A) < \infty$ and $\{v \neq \widehat{u}\}\Subset B_\varrho(x)$, so that $v+u_0$ is admissible for \eqref{1006181107} (notice that $v+u_0=u_0$ a.e.\ in $A \sm (\Omega\cup \dod)$ by \eqref{1006180918}).  
Being $u_0 \in W^{1,\infty}(\Omega';\Rn)$ and $u$ a solution to \eqref{1006181107}, we have that $\widehat{u}\in GSBD^2(A)$ with $\widehat{u}=0$ a.e.\ in $A \sm (\Omega\cup \dod)$ and  (by manipulating the quadratic forms in the minimality) 
\begin{equation}\label{0710181934}
\begin{split}
\int \limits_{B_\varrho(x)} \C e(\widehat{u}) \colon e(\widehat{u}) \dx  & + 2 \beta\, \hn(J_{\widehat{u}} \cap B_\varrho(x)) \leq \int \limits_{B_\varrho(x)} \C e(v) \colon e(v) \dx \\&  \hspace{1em} + 2 \int \limits_{B_\varrho(x)} \C e(u_0) \colon e(v-\widehat{u}) \dx + 2 \beta\, \hn(J_v \cap B_\varrho(x)) \,.
\end{split}
\end{equation}
We have that 
\begin{equation}\label{1006181324}
\begin{split}
\int \limits_{B_\varrho(x)}&  \C e(u_0) \colon e(v-\widehat{u}) \dx  \leq C_{\C} \|e(u_0)\|_{L^2(B_\varrho(x))} \|e(v-\widehat{u})\|_{L^2(B_\varrho(x))} 
\\& \hspace{-1em}\leq C_{\C}  \Big( \|e(u_0)\|_{L^2(B_\varrho(x))} \|e(\widehat{u})\|_{L^2(B_\varrho(x))}  +  \frac{\|e(u_0)\|_{L^2(B_\varrho(x))}^2 }{ 2\varepsilon^2}+ \frac{\varepsilon^2}{2} \|e(v)\|_{L^2(B_\varrho(x))}^2  \Big)  \,.
\end{split}
\end{equation}
Now
\begin{equation}\label{1006181217}
\|e(u_0)\|_{L^2(B_\varrho(x))} \leq C_{\|e(u_0)\|_\infty} \varrho^{\frac{n}{2}} \,,
\end{equation}
since $u_0\in W^{1,\infty}(\Omega';\Rn)$, and
\begin{equation}\label{1006181325}
\|e(\widehat{u}  + u_0 )\|^2_{L^2(B_\varrho(x))} + 2 \beta\,\hn(J_{\widehat{u}}\cap B_\varrho(x)) \leq  \tilde{c}_0 \varrho^{n{-}1}\,,
\end{equation}
with  $\tilde{c}_0=\tilde{c_0}(n,\,\C,\, \|e(u_0)\|_\infty,\, \wt L)$,  by comparing in \eqref{1006181107} the functional evaluated in $\widehat{u} + u_0 $ and in $((\widehat{u} + u_0) \chi_{A\sm (\Omega\cap B_\varrho(x))}  ) $ (cf.\ \cite[Lemma~3.10]{BabGia14},  in particular $\|e(u_0)\|_\infty$ and the fact that $\varrho\leq 1$ are used to possibly control $\|e(u_0)\|_{L^2(B_\varrho(x)\sm \Omega)}$ in terms of $\varrho^{n-1}$). 
By \eqref{0710181934} it follows that
\begin{equation*}
\begin{split}
G_0(\widehat{u}, B_\varrho(x)) \leq  G_0(v, B_\varrho(x)) + 2 \int \limits_{B_\varrho(x)} \C e(u_0) \colon e(v-\widehat{u}) \dx\,,
\end{split}
\end{equation*}
and collecting \eqref{1006181324} with $\varepsilon:=\varrho^{1/4}$, \eqref{1006181217}, \eqref{1006181325}, we get
\begin{equation*}
\begin{split}
G_0(\widehat{u}, B_\varrho(x)) \leq (1 + C_{\C}\, \varrho^{1/2})\, G_0(v, B_\varrho(x)) + C_{(\tilde{c}_0, \,\|e(u_0)\|_\infty)} \varrho^{n-\frac{1}{2}}\,.
\end{split}
\end{equation*}
Taking the infimum with respect to $v$  (admissible in the minimum problem \eqref{1006180922}, for $B_\varrho(x)$ in place of $A$)  we get
\begin{equation*}
\mathrm{Dev\,}(\widehat{u}, B_\varrho(x)) \leq \omega \varrho^{n-\frac{1}{2}}\,,
\end{equation*}
since $\m(\widehat{u}, B_\varrho(x)) \leq G_0(\widehat{u}, B_\varrho(x)) \leq \tilde{c_0}\varrho^{n-1}$, by \eqref{1006180922}, \eqref{1006181217}  and \eqref{1006181325}.   This concludes the proof. 
\end{proof}
\begin{remark}\label{rem:1006181350}
Notice that in Proposition~\ref{prop:1006181116} we have employed $u_0\in W^{1,\infty}$ in \eqref{1006181217}; it would be enough to require $u_0 \in W^{1,p}(\Omega';\Rn)$ with $n\frac{p-2}{2p}= n{-}1 + \widehat{\eta}$, for some $\widehat{\eta}>0$ to get that $\|e(u_0)\|_{L^2(B_\varrho(x))} \leq C_{\|e(u_0)\|_p} \varrho^{\frac{n-1+\widehat{\eta}}{2}}$ and that $\widehat{u}$ is a $(\omega, s)$-quasi-minimser, with $s$ depending on $\widehat{\eta}$.
\end{remark}

We now start the proof of regularity results for quasi-minimisers of $G_0$. We require that $\dod$ is of class $C^1$ to guarantee that $\dod$ converges to an hyperplane in the blow-up near any $x\in \dod$. The first lemma is a decay estimate for $G_0$, that holds for quasi-minimisers with small density of jump. The point is to made quantitative these smallness, and uniform with respect to the (sufficiently small) balls. The proof of the lemma is based on the results of the previous sections, and follows the structure of \cite[Lemma~6.6]{BabGia14}. 

\begin{remark}\label{rem:sceltaeta}
The constants in the following lemma depend also on a small parameter $\eta$, and the estimates are obtained for balls $B$ such that $\mathrm{dist\,}(B, \partial(\dod))>\eta$, where $\partial(\dod)$ is the boundary of $\dod$ in the relative topology of $\dom$. The parameter $\eta$ is employed only to guarantee that $\dod$ converges to an hyperplane also in blow-ups around points $x_h$ tending to $\ol x \in \partial(\dod)$. Such a property is ensured without the introduction of any $\eta$ if $\dod$ is uniformly of class $C^1$ up to $\ol{\dod}$, which is true for instance if $\dod$ is compactly contained in an open subset of $\dom$ of class $C^1$ (all these topological notions refer to the relative topology of $\dom$). In this case also the estimates in Theorem~\ref{teo:Teor3.4BG} are independent of $\eta$.
\end{remark}

\begin{lemma}\label{le:6.6BG} Let $\dod$ of class $C^1$.
Let $C_0$, $C'_0$ be the constants in Theorems~\ref{teo:ellreg1} and \ref{teo:ellreg2}, respectively. For every $\tau \in (0,1)$ and $\eta>0$ there exist positive constants $\varepsilon$, $\theta$ and $r$, depending on $\tau$ and $\eta$, such that
\begin{equation}\label{1006181802}
G_0(u, B_{\tau \varrho}(x))\leq 2 \max\{4^n, C_0, C_0'\}\, \tau^n\, G_0(u, B_\varrho(x))
\end{equation}
for every $B_\varrho(x)\subset \Omega'$ with $x\in \Omega\cup \dod$, $\varrho\leq r$,  $\mathrm{dist\,}(B_\varrho(x), \partial(\dod))>\eta$,  and for every $u\in GSBD^2(B_\varrho(x))$ with $u=0$ a.e.\ in $B_\varrho(x) \sm (\Omega\cup \dod) $ and
\begin{equation*}
\hn(J_u \cap B_\varrho(x)) \leq \varepsilon \varrho^{n-1}\,,\qquad \mathrm{Dev\,}(u, B_\varrho(x)) \leq \theta \, G_0(u, B_\varrho(x))\,.
\end{equation*}
\end{lemma}

\begin{proof}
If $\tau \geq 1/4$, then
\begin{equation*}
4^n\, \tau^n\, G_0(u, B_\varrho(x)) \geq G_0(u, B_\varrho(x)) \geq G_0(u, B_{\tau\varrho}(x))\,,
\end{equation*}
so \eqref{1006181802} follows. Then let $\tau < 1/4$.

We argue by contradiction, assuming that there exist $\tau < 1/4$  and $\eta>0$  such that there are sequences
\begin{equation*}
\varepsilon_h\,,\ \theta_h\,,\ r_h \to 0\,,
\end{equation*}
$B_{\varrho_h}(x_h)\subset \Omega'$ with $x_h \in \Omega\cup \dod$, $\varrho_h \leq r_h$,  $\mathrm{dist\,}(B_{\varrho_h}(x_h),\partial(\dod))>\eta$,  and $u_h \in GSBD^2(B_{\varrho_h}(x_h))$ with $u_h=0$ a.e.\ in $B_{\varrho_h}(x_h) \sm (\Omega \cup \dod)$,
\begin{equation*}
\hn(J_{u_h} \cap B_{\varrho_h}(x_h))=\varepsilon_h \varrho_h^{n-1}\,,\qquad \mathrm{Dev\,}(u_h, B_{\varrho_h}(x_h)) = \theta_h\, G_0(u_h, B_{\varrho_h}(x_h))\,,
\end{equation*} 
and
\begin{equation*}
G_0(u_h, B_{\tau \varrho_h}(x_h)) >  2 \max\{4^n, C_0, C_0'\}\, \tau^n\, G_0(u_h, B_{\varrho_h}(x_h))\,.
\end{equation*}
As usual  (see e.g.\ \cite{DeGCarLea}),  we rescale introducing the functions
 \begin{equation*}
 v_h(y):= \sqrt{\frac{\varrho_h^n}{G_0(u_h,   B_{\varrho_h}(x_h))} } \cdot \ \frac{u_h(x_h + \varrho_h y)}{\varrho_h} \qquad \text{for }y \in B_1\,,
 \end{equation*}
 and we call
 \begin{equation*}
 D_h:= \{y\in B_1 \colon x_h + \varrho_h y \in  B_{\varrho_h}(x_h)  \sm   (\Omega \cup \dod) \}\,.
 \end{equation*}
 Up to a subsequence $x_h \to \ol x \in \Omega\cup \dod$, since $\eta>0$.  If $D_h \neq \emptyset$ we have, thanks to \cite[Lemma~6.4]{BabGia14}, that, up to a futher subsequence, there is a coordinate system such that
 \begin{equation*}
 D_h := \{x=(x',x_n) \in B_1 \colon x_n \leq g_h(x')\}
 \end{equation*}
 for suitable $g_h\in C^1(\R^{n-1})$ with $g_h \to - \gamma$ locally uniformly, and $\gamma\in [0,1]$.  In this case one has $\ol x \in \dod$. 
 Moreover, the rescaling gives that
 \begin{equation}\label{1006181943}
 G_{D_h}(v_h, c_h, B_1)=1\,,\qquad \mathrm{Dev\,}(v_h, c_h, B_1)=\theta_h\,, \qquad\hn(J_{v_h} \cap B_1)=\varepsilon_h\,,
\end{equation}  
for 
\begin{equation*}
c_h:= \frac{\beta \varrho_h^{n-1}}{G_0(u_h, B_{\varrho_h}(x_h))}\,,
\end{equation*}
 and
 \begin{equation}\label{1006181945}
G_{D_h}(v_h, c_h, B_\tau) > 2 \max\{4^n, C_0, C_0'\}\, \tau^n\,.
 \end{equation}
 We consider first the case where $\gamma$, obtained as the limit of $-g_h$, is in $(1/2, 1]$,  which is as the standard case in \cite{ChaConIur17}.  Notice that the case $\gamma=1$ corresponds to $B_1 \sm H_\gamma=\emptyset$, with $H_\gamma$ as in \eqref{1006181950}: we assume then $\gamma=1$ also in the case that $D_h = \emptyset$ for every $h$. We apply \cite[Theorem~4]{ChaConIur17}, in the set $B_{1/2}$ with $k_h=0$, $\beta_h=c_h$ to the functions $v_h$ (the assumptions are satisfied by \eqref{1006181943}, that in particular holds with inequalities for $B_{1/2}$ in place of $B_1$): then there are $v\in H^1(B_{1/2};\Rn)$ and $a_h$ affine with $e(a_h)=0$ such that (up to a subsequence, not relabelled)
 \begin{equation*}
 v_h - a_h \to v \qquad \text{a.e.\ in } B_{1/2}
 \end{equation*}
 and $v$ is a local minimiser of $E_{0, \gamma}(\cdot, B_{1/2})$ with
  \begin{equation*}
 \int \limits_{B_{\tau}} \C e(v) \colon e(v) \dx = \lim_{h\to \infty} G_0(v_h, c_h, B_{\tau}) = \lim_{h\to \infty} G_{D_h}(v_h, c_h, B_{\tau})\,,
\end{equation*}  
and the same holds for every $\tilde{\tau}\leq 1/2$ in place of $\tau$ using that $D_h \cap B_{1/2}=\emptyset$ for every $h$ large enough. In particular, taking $\tilde{\tau}=1/2$, \eqref{1006181943} implies $\int_{B_{1/2}} \C e(v) \colon e(v) \dx\leq 1$.
Now Theorem~\ref{teo:ellreg1} gives (recall $\tau < 1/4$)
\begin{equation*}
\lim_{h \to \infty} G_{D_h}(v_h, c_h, B_{\tau}) =  \int \limits_{B_{\tau}} \C e(v) \colon e(v) \dx \leq C_0 \, \tau^n\,,
\end{equation*}
and this contradicts \eqref{1006181945}.

On the other hand, if $\gamma\in [0, 1/2]$ we apply Theorem~\ref{thm:Teor4} (again, the assumptions are satisfied by \eqref{1006181943}): there are $v\in H^1(B_{1};\Rn)$ local minimiser of $E_{0, \gamma}(\cdot, B_1)$ such that (up to a subsequence, not relabelled)
 \begin{equation*}
 \begin{split}
 v_h &\to v \qquad \text{a.e.\ in } B_1\,,\\
 \lim_{h\to \infty} G_{D_h}(v_h, c_h,\, & B_{\tau}) = \int \limits_{B_{\tau}} \C e(v) \colon e(v) \dx  \,,
 \end{split}
\end{equation*}  
and this holds also for every $\tilde{\tau}\leq R_0 =  3/4$,  so that $\int_{ B_{3/4} } \C e(v) \colon e(v) \dx\leq 1$, by \eqref{1006181943}. Employing Theorem~\ref{teo:ellreg2}  (it is enough that $R_0> 2/3$)  we get
\begin{equation*}
\lim_{h \to \infty} G_{D_h}(v_h, c_h, B_{\tau}) =  \int \limits_{B_{\tau}} \C e(v) \colon e(v) \dx \leq C'_0 \, \tau^n\,,
\end{equation*}
in contradiction to \eqref{1006181945}.
\end{proof}
The following theorem is a general weak regularity result for all $(\omega,s)$-quasi-minimisers of $G_0(\cdot,A)$ (see Definition~\ref{def:quasiMin}).
\begin{theorem}\label{teo:Teor3.4BG}
Let $\dod$ of class $C^1$, and $A\subset \Omega'$ be an open set and $u\in GSBD^2(A)$ be a $(\omega,s)$-quasi-minimiser of $G_0(\cdot,A)$. Then  for every $\eta>0$,  there exist $\theta_0$ and $\varrho_0>0$, depending only on $n$, $\C$, $\beta$, $s$, $\omega$,  $\eta$,  such that
\begin{equation}\label{1006180944}
\hn(J_u \cap B_\varrho(x))\geq \theta_0\, \varrho^{n-1}
\end{equation}
for all balls $B_\varrho(x) \subset A$ with $x\in \overline{J^*_u}$, $\varrho\leq \varrho_0$,  and $\mathrm{dist\,}(B_\varrho(x), \partial(\dod))>\eta$, 
where
\begin{equation}\label{1106181722}
J^*_u:=\left \{ x \in J_u \colon \lim_{\varrho \to 0} \frac{\hn(J_u \cap B_\varrho(x))}{\omega_{n-1} \varrho^{n-1}}=1 \right\}\,,
\end{equation}
with $\omega_{n-1}$ the $(n{-}1)$-dimensional Lebesgue measure of the unit ball in $\R^{n-1}$.
\end{theorem}

\begin{proof}
Consider the set $J^*_u$ in \eqref{1106181722}. We have $\hn(J_u \sm J^*_u)=0$, since $J_u$ is countably $(n{-}1)$-rectifiable. Then we can follow exactly \cite[Theorem~3.4]{BabGia14} with $J^*_u$, $G$, $2 \max\{4^n, C_0, C_0'\}$  instead of $S_u \sm I$, $F$, $C_1$ therein, respectively (notice that equation (6.13) in \cite[Theorem~3.4]{BabGia14} holds also for $e(u)$ in place of $\nabla u$). It is enough to employ Lemma~\ref{le:6.6BG} in place of \cite[Lemma~6.6]{BabGia14}.
\end{proof}
We are now in the position to prove the main result of the paper, that is specialised in Corollary~\ref{cor:1106181816} obtaining the desired regularity for solutions to \eqref{1006181103}.
\begin{theorem}\label{teo:regDirLoc}
Let $\dod$ of class $C^1$,  $A\subset \Omega'$,  and $u \in GSBD^2(A)$ be a solution to \eqref{1006181107}. Then
\begin{equation}\label{1006180945}
\hn\big(A \cap (\overline{ J_u }\sm  J_u ) \big)=0\,,
\end{equation}
and  (up to passing to a precise representative $\wt u$, equal to $u$ $\mathcal{L}^n$-a.e.)
\begin{equation}\label{1106181747}
u \in  C^{ \infty }(A \cap \Omega \sm \ol  J_u ; \Rn)  \cap  C(A \sm \ol  J_u ; \Rn) \,. 
\end{equation}
 Moreover, for any $U \Subset A \sm \ol  J_u $ connected, there is $C_U>0$, depending on $U$, such that for any $x$, $y \in U$
\begin{equation}\label{1011180846}
\Big| \big(u(x)-u(y)\big) \cdot \frac{x-y}{|x-y|} \Big| \leq C_U |x-y|^{1/2}\,, 
\end{equation}
 and for any $x_0 \in \dod \cap A_\eta \sm \ol  J_u $
\begin{equation}\label{2010182027}
|u_0(x_0)- u(x)| \leq C |x-x_0|^{1/2}\,,\quad\text{for any $x \in A_\eta \sm \ol  J_u $ with }|x-x_0|\leq r_{x_0}\,,
\end{equation}
for suitable $r_{x_0}>0$ depending on $x_0$ and $C>0$ depending only on $n$, $L$, and on the parameters of the Griffith functional. 
\end{theorem} 
\begin{proof}  We divide the proof into  two  parts. Let us first fix $\eta >0$ and denote 
\begin{equation*}
A_\eta:= A \cap \{\mathrm{dist}(\cdot, \partial(\dod))>\eta\}\,.
\end{equation*}
 In the proof we work with $J^*_u$ (cf.\ \eqref{1106181722}), obtaining the statement for this set.  Then, \eqref{1106181747} for $J^*_u$ gives 
\[
\hn(\ol{J_u} \sm J^*_u)=0\,,
\]
so in particular $J_u$ is essentially closed and equal to $J^*_u$, up to a $\hn$-negligible set, and we can express all in terms of $J_u$.

\paragraph{\textbf{Part 1. Essential closedness of $J^*_u$ and internal regularity.}}
 By Proposition~\ref{prop:1006181116} (applied for $A_\eta$) and  Theorem~\ref{teo:Teor3.4BG} it follows that for any $x \in \overline{J^*_u} \cap A_\eta$  the upper $(n{-}1)$-dimensional  density of the measure $\hn \mres J_u$ at $x$  (cf.\ \cite[Definition~2.55]{AFP}),  that is 
\begin{equation*}
\limsup_{\varrho \to 0} \frac{\hn(J_u \cap B_\varrho(x))}{\omega_{n-1} \varrho^{n-1}}\,,
\end{equation*} 
is greater than $\frac{\theta_0}{\omega_{n-1}}$. Therefore we may employ \cite[Theorem~2.56]{AFP} with $k=n{-}1$, $\mu= \hn \mres J_u$, $t=\frac{\theta_0}{\omega_{n-1}}$, and $B= A_\eta \cap  \overline{J^*_u}$,  to get that $ \hn(A_\eta \cap \overline{J^*_u} \sm J_u)=0$.  
We notice that 
\begin{equation}\label{3110181016}
 \hn(A_\eta \cap \overline{J^*_u} \sm J^*_u)=0\,,
\end{equation}
since $\hn(J_u \sm J^*_u)=0$, being $J_u$ countably $(n{-}1)$-rectifiable.

Since $\hn(J_u \cap A_\eta \sm \overline{J^*_u}) \subset \hn(J_u \sm J^*_u) =0$, by the slicing properties in the definition of $G(S)BD$ we get that $u\in H^1_{\mathrm{loc}}(A_\eta \sm \overline{J^*_u})$.  By regularity of solutions to $\mathrm{div} (\C e(u))=0$ in open sets (see e.g.\ \cite[Theorem~6.2.1]{Mor66}) it follows that 
\begin{equation}\label{1106181810''}
u \in C^{ \infty }(A_\eta \cap \Omega \sm \overline{J^*_u};\Rn)\,. 
\end{equation}

\paragraph{\textbf{Part 2. Continuity up to $\dod$.}}
We assume that $A \sm \Omega \neq \emptyset$ and prove that 
\begin{equation}\label{2110180135}
u \in C(A_\eta \sm \overline{J^*_u}; \Rn)\,.
\end{equation}
Let $L$ be the Lipschitz constant of $\dod \cap A_\eta$ (regarded as the common boundary between $\Omega \cap A_\eta$ and $\big(\Omega' \sm (\Omega \cup \dod)\big) \cap A_\eta$).

Since $u-u_0$ is a $(\omega, 1/2)$ quasi-minimiser of $G_0(\cdot, A_{ \eta })$ (see Definition~\ref{def:quasiMin} and Proposition~\ref{prop:1006181116}), for any $B_\varrho(x) \subset A_\eta$ we have (see \eqref{1006181325})
\begin{equation}\label{2010182210}
\|e(u)\|^2_{L^2(B_\varrho(x))} \leq \tilde{c}_0 \varrho^{n-1}\,,
\end{equation}
with $\tilde{c}_0$ depending on $n$, $\C$, and $L$.
{
  Given $x,\varrho$ with $B_\varrho(x)\subset A_\eta\sm\ol{J^*_u}$  and $\varrho\leq 1$,  there exists an infinitesimal
  rigid motion, that is an affine function
  \[
    a_{x,\varrho} (y) = u_{x,\varrho} + S_{x,\varrho}(y-x),
  \]
  where $u_{x,\varrho}$ is the average of $u$ over $B_{\varrho}(y)$ and $S_{x,\varrho}$ a linear skew-symmetric map,
  such that
\begin{equation}\label{2010182321}
\int \limits_{B_\varrho(x)} |u -{a}_{x, \varrho}|^2 \dz \leq C \varrho^2 \int \limits_{B_\varrho(x)} |e(u)|^2 \dz \leq C \varrho^{n+1}\,, \quad\text{for any $\varrho \leq r$},
\end{equation}
thanks to the Poincaré-Korn inequality and \eqref{2010182210}.
}

Let us fix $x_0 \in \partial_D\Omega \cap A_\eta \sm \ol J^*_u$.
 In the following we show that $u$ admits a precise representative $\wt u$ (namely, $\wt u=u$ a.e.) defined everywhere in $A_\eta \sm \ol J^*_u$  and  prove  \eqref{1011180846} and \eqref{2010182027}, 
arguing in the spirit of
Campanato's theorem \cite{Camp63} (see also \cite[Theorem~7.51]{AFP}). In the rest of the proof $C$ will denote a constant depending only on $n$, $L$, 
and on the parameters of the Griffith functional.

By the regularity of $\partial_D\Omega \cap A_\eta$, we find a hyperplane $H_0$, with normal $\nu_0$, a $L$-Lipschitz function $l_0 \colon H_0\to \R$, and $r_0$, $h_0>0$ such that, for 
\[C_0:=\{x+\nu_0\,  y \colon x \in B_{r_0}(x_0) \cap H_0, \, |y| < h_0 \}\,,\] we have  
\[
\partial\Omega \cap C_0 = \{ x+ \nu_0 \, y \colon x \in B_{r_0}(x_0) \cap H_0,\, l_0(x)=y\},\]
and $\Omega \cap C_0 = \{ x+ \nu_0 \, y \colon x \in B_{r_0}(x_0) \cap H_0,\, -h_0 < y <l_0(x)\}$.
Moreover, we may assume that $B_{r_0}(x_0) \subset A_\eta \sm \ol J^*_u$.

\paragraph{\bf Step 1.}  First, let us prove that $u$ admits a precise representative in any point of $A_\eta \sm \ol J^*_u$ and estimate its distance from the average of $u$ in small balls centered in the point.
\def\mwt{}
Let $x,r$ with with $B_r(x)\subset A_\eta$. We claim that for any $\varrho \leq r$,
\begin{equation}\label{2110180017}
\| \mwt{a}_{x, \varrho} - \mwt{a}_{x, \varrho/2}\|_{L^\infty( B_{\varrho/2}(x)) } \leq C \varrho^{1/2}\,.
\end{equation}
Indeed, { as $|\mwt{a}_{x, \varrho} - \mwt{a}_{x, \varrho/2}|^2\leq 2 |u-\mwt{a}_{x, \varrho}|^2 + 2 |u-\mwt{a}_{x, \varrho/2}|^2$~a.e., using \eqref{2010182321}} for $\varrho$ and $\varrho/2$ we deduce
\begin{equation*}
\int \limits_{B_{\varrho/2}(x)} |\mwt{a}_{x, \varrho} - \mwt{a}_{x, \varrho/2}|^2 \dz \leq C \varrho^{n+1}\,,
\end{equation*}
and then 
\begin{equation}\label{3110181812}
 \| \mwt{a}_{x, \varrho} - \mwt{a}_{x, \varrho/2}\|_{L^\infty( B_{\varrho/2}(x)) } \leq C \Big(\varrho^{-n} \int \limits_{B_{\varrho/2}(x)} |\mwt{a}_{x, \varrho} - \mwt{a}_{x, \varrho/2}|^2 \dz\Big)^{1/2}  \leq C \varrho^{1/2}\,,
\end{equation}
so \eqref{2110180017} follows. Notice that in \eqref{3110181812} we have used the fact that $\mwt{a}_{x, \varrho} - \mwt{a}_{x, \varrho/2}$ is affine, and that for any $a \colon B_{\varrho/2}(x) \to \Rn$ affine, letting $a_\varrho(y):=a\Big(\frac{2 y}{\varrho}-x \Big) $, it holds
\[
\|a\|_{L^\infty(B_{\varrho/2}(x))} = \|a_\varrho\|_{L^\infty(B_1)} \leq \ol C_n \|a_\varrho\|_{L^2(B_1)}= \ol C_n \Big((\varrho/2)^{-n}\int \limits_{B_{\varrho/2}(x)} |a|^2 \dx\Big)^{1/2}\,,
\]
for $\ol C_n$ depending only on $n$.

From \eqref{2110180017} we get for any $i \in \N$  (formally replacing $\varrho$ with $2^{-i}\varrho$) 
 \begin{equation*}
\| \mwt{a}_{x, 2^{-i}\varrho} - \mwt{a}_{x, 2^{-i}\varrho/2}\|_{L^\infty(B_{2^{-i}\varrho/2}(x))} \leq C \, 2^{-i/2}\varrho^{1/2}\,.
\end{equation*}
{ We easily deduce that $u_{x,\varrho}=a_{x,\varrho}(x)$ is a Cauchy sequence
  so that there exists the limit $\wt{u}(x):=\lim_{\varrho \to 0} u_{x, \varrho}$. Moreover,}
\begin{equation}\label{2110180124}
\| \mwt{a}_{x, \varrho} - \mwt{a}_{x, 2^{-h}\varrho}\|_{L^\infty(B_{2^{-h}\varrho}(x))}  \leq \sum_{i=0}^{h-1} \| \mwt{a}_{x, 2^{-i}\varrho} - \mwt{a}_{x, 2^{-i}\varrho/2}\|_{L^\infty(B_{2^{-h}\varrho}(x))}  \leq  C \, \varrho^{1/2}\,,
\end{equation}
{ and we find in addition that}
\begin{equation}\label{2110180125}
| \mwt{a}_{x, \varrho}(x) - \wt{u}(x)| \leq C \, \varrho^{1/2}\,.
\end{equation}
{ In particular, we observe that any point away from $\ol{J_u^*}$ is a Lebesgue point.}

{ \paragraph{\textbf{Step~2.}} We now prove \eqref{1011180846}.}
Fix $U \Subset A \sm \ol J^*_u$ connected, so that there are $\ol r$, $\eta>0$ such that $B_{2\ol r}(z) \subset A_\eta \sm \ol J^*_u$ for any $z \in U$. Fix also $x$, $y \in U$, with $|x-y|=:r \leq \ol r $.

{ We have:
\begin{equation*}
\int \limits_{B_r(\frac{x+y}{2})} (\mwt a_{x, 2r} - \mwt a_{y, 2r}) \,\mathrm{d}z= |B_r| \Big[\big( \mwt a_{x, 2r}(x) - \mwt a_{y, 2r} (y) \big) + S_{x,2r}\Big(\frac{y+x}{2}-x\Big) - S_{y,2r}\Big(\frac{x+y}{2}-y\Big)\Big]\,.
\end{equation*}
 Moreover,  since the matrices are skew-symmetric,
\begin{equation*}\label{1011181025}
\Big[S_{x,2r}\Big(\frac{y-x}{2}\Big) - S_{y,2r}\Big(\frac{x-y}{2}\Big) \Big] \cdot (x-y)=0\,.
\end{equation*} 
Then, we use}
\begin{equation*}\label{1011181026}
\begin{split}
\Big| r^{-n} \hspace{-1em}\int \limits_{B_r(\frac{x+y}{2})} \hspace{-1em} (\mwt a_{x, 2r} - \mwt a_{y, 2r}) \,\mathrm{d}z \Big| &\leq \wt C_n \Big(r^{-n} \hspace{-1em}\int \limits_{B_r(\frac{x+y}{2})} \hspace{-1em} |\mwt a_{x, 2r} - \mwt a_{y, 2r}|^2 \,\mathrm{d}z \Big)^{1/2} \\ 
&\leq C \Big( r^{-n} \hspace{-1em} \int \limits_{B_{2r}(x)} \hspace{-0.7em}|u- \mwt a_{x, 2r}|^2 \dz\Big)^{1/2} \hspace{-0.7em} + C \Big( r^{-n} \hspace{-1em} \int \limits_{B_{2r}(y)} \hspace{-0.7em} |u- \mwt a_{y, 2r}|^2 \dz \Big)^{1/2} \hspace{-0.7em} \leq C r^{1/2}\,,
\end{split}
\end{equation*}
by \eqref{2010182321}. 
Collecting the relations above, and recalling \eqref{2110180125}, we deduce \eqref{1011180846}, under the assumption that $|x-y| \leq \ol r$.
Then \eqref{1011180846} is extended to general $x$, $y \in U$ by employing the connectedness of $U$ (cf.\ \cite[Teorema~I.2]{Camp63}). 

\paragraph{\bf Step 3.}
{
  We now prove \eqref{2010182027} for
\[
r_{x_0}=\frac{\min\{h_0 , r_0\}}{2}\,.
\]
}
We fix $x \in B_{r_{x_0}}(x_0)$, and for $r:= |x-x_0|\leq r_{x_0}$ let $y_0 := x_0 + \nu_0\, r \in \Omega' \sm \ol\Omega$. Then, { assuming without loss of generality that $L\ge 1$,
\begin{equation}\label{3110181243}
  B_{\frac{r}{L}}(y_0) 
  \subset \Omega' \sm \ol \Omega\,.
\end{equation}
}
{ We use \eqref{1011180846}
for a suitable $U\supset B_{r_{x_0}}(x_0)\cup B_{\frac{r}{L}(y_0)}$.}
Let us denote 
$y_i:= y_0 + (r/L) e_i$, for $i\in \{1, \dots, n\}$. Then, by \eqref{1011180846}, \eqref{3110181243}, and since $u_0$ is Lipschitz and $\wt u(y_i)= u_0(y_i)$, it holds that
\begin{equation*}
\begin{split}
\Big| \big( \wt u(x) - \wt u(y_0) \big) \cdot \frac{x-y_i}{r} \Big| &\leq \Big| \big( \wt u(x) - \wt u(y_i) \big) \cdot \frac{x-y_i}{r} \Big| + C \, r/L \\&\leq C_L \Big| \big( \wt u(x) - \wt u(y_i) \big) \cdot \frac{x-y_i}{|x-y_i|} \Big| + C \, r/L \leq C r^{1/2}
\end{split}
\end{equation*}
for any $i \in\{0, 1, \dots, n\}$. By combining these inequalities, we get 
\begin{equation*}
\Big| \big( \wt u(x) - \wt u(y_0) \big) \cdot e_i \Big| \leq C  \sqrt{r} 
\end{equation*}
for any $i \in \{1, \dots, n\}$, and then \eqref{2010182027}, using again that $u_0$ is Lipschitz and $\wt u ( y_0)= u_0(y_0)$.

 By the arbitrariness of $\eta>0$,  \eqref{1006180945} and \eqref{1106181747} follow  from \eqref{3110181016}, \eqref{1106181810''}, and \eqref{2110180135}. 
 The proof is then concluded.

\end{proof}

\begin{corollary}\label{cor:1106181816}
Let $\dod$ of class $C^1$, and $u \in GSBD^2(\Omega')$ be a solution to \eqref{1006181103}. Then 
\begin{equation*}\label{1106181819}
  J_u   \cup K \subset \Omega\cup \dod \text{ is (essentially) closed in the topology of } \Omega'\,,
\end{equation*}
and 
\begin{equation*}\label{1106181821}
u \in C^{ \infty }(\Omega \sm (   {J_u}  \cup K); \Rn) \cap C(\Omega' \sm (    {J_u}   \cup K); \Rn)\,.
\end{equation*}
\end{corollary}
\begin{proof}
 It is enough to apply Theorem~\ref{teo:regDirLoc} with
\begin{equation*}
A= \Omega' \sm K\,.
\end{equation*}
It is immediate that $J_u \subset \Omega\cup \dod$, since $u=u_0$ a.e.\ in $\Omega' \sm (\Omega\cup \dod)$.
\end{proof}

\begin{remark}
By Corollary~\ref{cor:1106181816}, a strong solution to the Dirichlet problem for the Griffith energy \eqref{enGriffith} is $(u, \Gamma)$, where $u$ is a weak minimiser and $\Gamma= \ol {J_u}$. 
\end{remark}

\begin{appendices}

\section{Appendix}\label{sec:app}
In this appendix we deal with
a regularity result for solutions of elliptic equations with Dirichlet boundary conditions.
By \cite[Theorem~4.18, (i)]{McLean}, the estimate \eqref{0707182305} below is formally obtained  (with the notations in \cite{McLean}, in particular $\gamma$ represents there the trace operator) by taking $G_1=B_{3 R_0/4}$, $G_2=B_{R_0}$, $\mathcal{P}u=\mathrm{div\,}\C e(u)$, $f=0$, $\gamma u=0$. However, since the dependence of $C'_{0,m}$ from the other relevant known constant is not clearly specified in \cite{McLean}, and it is very important for Theorem~\ref{teo:ellreg2} and its consequences,
we give an outline of the proof.
 We refer to the notation of Section~\ref{sec:Sec2}, in particular recall \eqref{1006181950} and \eqref{0906181125}.

 \begin{theorem}\label{teo:ellreg2App}
 Let $\gamma\in [0, 1/2]$, $u\in H^1(B_1;\Rn)$ be a local minimiser of $E_{0,\gamma}(\cdot, B_1)$, and $R_0<1$ be such that $\frac{3}{4}R_0 > \gamma$. Then for every $m\in \N$ and $\varrho\leq R_0$ there exists $C_{0,m}'$ depending on $\C$, $m$, and $R_0$, such that 
 \begin{equation}\label{0707182305}
 \|u\|_{H^m(B_{3R_0/4};\Rn)}\leq  C'_{0,m} \|e(u)\|_{L^2(B_{R_0};\Mnn)}\,.
 \end{equation}
 \end{theorem}

  \begin{proof}
  First let us prove that for any $\varrho < R_0$ it holds
 \begin{equation}\label{0807180114}
 \|u\|_{H^2(B_\varrho)} \leq C_{\C} (R_0- \varrho)^{-2} \|e(u)\|_{L^2(B_{R_0})}\,.
 \end{equation}
Let us fix $\varrho < R_0$ and take a cut-off function $\psi$ between $B_{\varrho}$ and $B_{R_0}$, that is $\psi \colon B_1 \to [0,1]$, $\psi \in C^\infty_c(B_{R_0})$, $\psi = 1$ in $B_{\varrho}$, such that
 \begin{equation}\label{0807182255}
 \| \nabla \psi\|^2_{L^\infty(B_1)} + \| \mathrm{D}^2 \psi\|_{L^\infty(B_1)} \leq C (R_0-\varrho)^{-2}\,,  
 \end{equation}
 for a universal constant $C$. For any $w \colon B_1 \to \R^s$, $s \geq 1$, and $x \in B_{1-h}$ we denote 
 \begin{equation*}
 \nabla_{l,h} w (x) := \frac{w(x+h e_l)- w(x)}{h}
 \end{equation*}
 the difference quotient in the direction $e_l$, where $e_l$ is the $l$-th element of the canonical basis of $\Rn$.
 Since $\psi u$ has compact support in $B_{R_0}$ and $\C$ has constant coefficients (with respect to $x$) we have that for $h$ small
 \begin{equation*}
 \int \limits_{B_{R_0}} \C\, \nabla_{l,h}(e(\psi u)) \colon e(v) \dx = -  \int \limits_{B_{R_0}} \C e(\psi u) \colon \nabla_{l,-h} (e(v)) \dx 
 \end{equation*}
 for any $v\in H^1(B_{R_0})$. Then for $h$ small and $v \in H^1_0(B_{R_0}; \Rn)$, $v=0$ in $B_{R_0}\sm H_\gamma$
 \begin{equation}\label{0807180055}
 \begin{split}
 \bigg|  & \int \limits_{B_{R_0}} \C\, \nabla_{l,h}(e(\psi u)) \colon e(v) \dx   \bigg| = \bigg|  \int \limits_{B_{R_0}} \C e(\psi u) \colon \nabla_{l,-h} (e(v)) \dx   \bigg|  \\& \leq \bigg| \int \limits_{B_{R_0}} \C  (\nabla \psi \odot u)  \colon e(\nabla_{l,-h} v) \dx \bigg|  = \bigg| \int \limits_{B_{R_0}} \C \,  \nabla_{l,h}(\nabla \psi \odot u)  \colon e(v) \dx \bigg| \\& \leq C_{\C} \Big( \| \mathrm{D}^2  \psi \|_{L^{\infty}(B_1)} \|u\|_{L^2(B_{R_0})} + \|\nabla \psi\|_{L^\infty(B_1)} \|e(u)\|_{L^2(B_{R_0})} \Big) \|e(v)\|_{L^2(B_{R_0})} \\& \leq C_{\C} (R_0- \varrho)^{-2} \|e(u)\|_{L^2(B_{R_0})}  \|e(v)\|_{L^2(B_{R_0})}\,,
 \end{split}
 \end{equation}
 where in the first inequality we have used that $e(\psi u) = \psi\, e(u) + \nabla \psi \odot u$ and the Euler equation for minimisers of \eqref{0906181125}
 \begin{equation*}
 \int \limits_{B_{R_0}} \C e(u) \colon e(v) \dx = 0 \quad \text{for any }v \in H^1_0(B_{R_0}; \Rn),\, v=0\text{ in }B_{R_0}\sm H_\gamma\,,
 \end{equation*}
 and the last inequality follows from \eqref{0807182255} plus Poincaré's and Korn's inequality for $H^1_0$ functions (cf.\ also below in \eqref{0807180111}).
 We now take, for $l=1, \dots, n{-}1$, $v:= \nabla_{l,h}(\psi u)$ as test function; indeed, by the form of $H_\gamma$ we have that\[\nabla_{l,h}(\psi u) \in H^1_0(B_{R_0}; \Rn),\quad \nabla_{l,h}(\psi u) =0\text{ in }B_{R_0}\sm H_\gamma\,.\]
 With this choice, we get
\begin{equation} \label{0807180111}
 \bigg|  \int \limits_{B_{R_0}} \C\, \nabla_{l,h}(e(\psi u)) \colon e(v) \dx   \bigg| \geq C_{\C} \| e(\nabla_{l,h}(\psi u))\|^2_{L^2(B_{R_0})} \geq C_{\C} \| \nabla_{l,h} u \|^2_{H^1(B_\varrho)} \,,
 \end{equation}
 since 
 \begin{equation*}\label{0807180120}
 \| \nabla_{l,h} u \|^2_{H^1(B_\varrho)}=\| \nabla_{l,h}(\psi u) \|^2_{H^1(B_{\varrho})}  \leq C \| e(\nabla_{l,h}(\psi u))\|^2_{L^2(B_{\varrho})}
 \end{equation*} 
for a universal constant $C$: this holds by the combination of Korn's inequality in $H^1_0(B_{\varrho})$ 
 \[
 \| \nabla\big(\nabla_{l,h}(\psi u) \big) \|^2_{L^2(B_\varrho)}  \leq 2 \| e(\nabla_{l,h}(\psi u))\|^2_{L^2(B_\varrho)}\,,
 \]
 and Poincaré's inequality in $H^1_0(B_{\varrho})$
 \[
 \| \nabla_{l,h}(\psi u) \|^2_{H^1(B_\varrho)}  \leq \big(1 + 9 R_0^2 / 16 \big) \| \nabla\big(\nabla_{l,h}(\psi u) \big) \|^2_{L^2(B_\varrho)}\,,
 \]
being $\varrho \leq \frac{3}{4}R_0$.
 
 As usual, to prove regularity of solutions to elliptic equations, the derivative $\partial_{n\,n} u$ is estimated by looking at the equation in weak form $\mathrm{div\,} \C e(u)=0$, that gives
 \begin{equation*}
 \| \partial_{n\,n} u\|_{L^2(B_\varrho)} \leq C_{\C} \big( \|u\|_{H^1(B_\varrho)} + \sum_{l=1}^{n{-}1} \| \partial_l u\|_{H^1(B_\varrho)} \big) \,.
 \end{equation*}
 Combining the estimate above with \eqref{0807180055} and \eqref{0807180111} (and standard properties of difference quotients), we obtain \eqref{0807180114}. Arguing in a similar way it is possible to show that if $\mathrm{div}\, \C e(w)=f$ in $B_{R_0} \cap H_\gamma$ and $w= 0$ in $B_{R_0} \sm H_\gamma$, then 
 \begin{equation}\label{0907180114}
 \|w\|_{H^2(B_\varrho)} \leq C_{\C} (R_0- \varrho)^{-2} \|e(w)\|_{L^2(B_{R_0})} + C_{\C} \|f\|_{L^2(B_{R_0})} \,.
 \end{equation}
 
 Now it is proven by induction that
 \begin{equation}\label{0807180135}
 \|u\|_{H^{m+1}(B_\varrho)} \leq C_{\C} (R_0- \varrho)^{-2m} \|e(u)\|_{L^2(B_{R_0})}\,,
 \end{equation}
the case $m=1$ being \eqref{0807180114}. For $l=1, \dots, n{-}1$ we have that $\mathrm{div}\, \C e(\partial_l u)$ is expressed in terms of derivatives of $u$ of order at most $m+1$ (cf.\ \cite[Lemma~4.13]{McLean}) and $\partial_l u= 0$ in $B_{R_0} \sm H_\gamma$, so that \eqref{0907180114} and  the induction assumption for $m$ give
 \begin{equation}\label{0807180201}
 \|\partial_l u\|_{H^{m+1}(B_\varrho)} \leq C_{\C} (R_0- \varrho)^{-2(m+1)} \|e(u)\|_{L^2(B_{R_0})}\,.
 \end{equation}
 It lasts to estimate $\partial^{m+2}_n u$, that is the derivative of $u$ taken $m+2$ times with respect to $e_n$. In order to do so, it is enough to apply $\partial^m_n$ to the explicit expression of $\partial_{n\,n}$ in terms of the other second order derivatives obtained from $\mathrm{div\,} \C e(u)=0$: then $\partial^{m+2}_n$ is a linear combination (trough combination of coefficients of $\C$) of the other derivatives of order $m+2$, already estimated in \eqref{0807180201}.
 We conclude \eqref{0707182305} by taking $\varrho= \frac{3}{4} R_0$ in \eqref{0807180135}.
 \end{proof}

\begin{remark}
From Theorem~\ref{teo:ellreg2App}, employing the Sobolev embedding $H^m \hookrightarrow C^1$ for any $m> 2 + n/2$ and recalling \eqref{eq:hpC}, we obtain Theorem~\ref{teo:ellreg2}.
\end{remark}

\end{appendices}

\bigskip
\noindent {\bf Acknowledgements.}
V.\ C.\ has been supported by a public grant as part of the \emph{Investissement d'avenir} project, reference ANR-11-LABX-0056-LMH, LabEx LMH, is currently funded by the Marie Sk\l odowska-Curie Standard European Fellowship No.\ 793018, and acknowledge the financial support from the Laboratory Ypatia of Mathematical Sciences LYSM and the Labo CMAP.


\end{document}